\documentclass{article}
\usepackage[utf8]{inputenc}
\usepackage{amsmath}
\usepackage{amssymb}
\usepackage{graphicx}
\usepackage{epstopdf}
\usepackage{amsthm}
\usepackage{hyperref}
\usepackage{amsfonts}
\usepackage{array}
\usepackage{amsthm}
\usepackage{tabularx}
\usepackage{url}
\usepackage{setspace}
\usepackage{geometry}
\usepackage{subfigure}
\usepackage{authblk}
\usepackage{xcolor}
\usepackage{multirow}
\usepackage{lipsum}
\usepackage[nottoc]{tocbibind}
\newtheoremstyle{break}
  {\topsep}{\topsep}%
  {\itshape}{}%
  {\bfseries}{}%
  {\newline}{}%
\theoremstyle{break}
\newtheorem{theorem}{Theorem}
\newtheorem{remark}{Remark}

\newtheorem{lemma}{Lemma}

\newtheorem{example}{Example}
\newtheorem{algorithm}{Algorithm}
\newcommand{\vertiii}[1]{{\left\vert\kern-0.25ex\left\vert\kern-0.25ex\left\vert #1
    \right\vert\kern-0.25ex\right\vert\kern-0.25ex\right\vert}}

\begin{document}
\title{Debiased and threshold ridge regression for linear model with heteroskedastic and dependent error}
\author{Yunyi Zhang and Dimitris N. Politis}
\maketitle
\abstract{Focusing on a high dimensional linear model $y = X\beta + \epsilon$ with dependent, non-stationary, and heteroskedastic errors, this paper applies the debiased and threshold ridge regression method that gives a consistent estimator for linear combinations of $\beta$; and derives a Gaussian approximation theorem for the estimator. Besides, it proposes a dependent wild bootstrap algorithm to construct the estimator's confidence intervals and perform hypothesis testing. Numerical experiments on the proposed estimator and the bootstrap algorithm show that they have favorable finite sample performance.

Research on a high dimensional linear model with dependent(non-stationary) errors is sparse, and our work should bring some new insights to this field.}
\section{Introduction}
Linear regression with independent and identically distributed(i.i.d.) errors is a fundamental topic in statistical inference. The classical setting assumes the dimension of parameters in a linear model is constant. Under this setting, research has been proposed on estimation, e.g., Zou and Hastie \cite{doi.org/10.1111/j.1467-9868.2005.00503.x} and Zou \cite{doi:10.1198/016214506000000735}; confidence intervals construction/hypothesis testing, e.g., Chatterjee and Lahiri \cite{doi:10.1198/jasa.2011.tm10159, 10.2307/41059185}; and prediction, e.g., Stine \cite{10.2307/2288570} and Zhang and Politis \cite{zhang2021bootstrap}. We also refer Seber and Lee \cite{seberLee} for an extensive introduction.

In reality, however, errors in a linear model can be dependent or have different distributions. For example,-as Vogelsang \cite{VOGELSANG2012303} and Petersen \cite{10.1093/rfs/hhn053} suggested-, heteroscedasticity, autocorrelation and spatial correlation present in panel data. In this case confidence intervals developed for i.i.d. errors may fail to capture the correct probability. New tools are developed to adapt to the non-i.i.d. errors. Andrews \cite{10.2307/2938229} and Kim and Sun \cite{KIM2011349} considered estimating the ordinary least square estimator's covariance matrix; Kelejian and Prucha \cite{KELEJIAN2007131} and Vogelsang \cite{VOGELSANG2012303} proposed test statistics; Sun and Wang \cite{doi:10.1080/07474938.2021.1874703} and Conley et al. \cite{Conley} worked on inference and hypothesis testing, etc. Despite great success, their works focus on the classical situation, i.e., the dimension of parameters is considered to be fixed and does not change as the sample size increases.

In the modern era, observations may have a comparable or even larger dimension than the number of samples. So the dimension of parameters cannot be considered as fixed. In order to perform consistent estimation, statisticians need to assume the underlying parameters are sparse(i.e., the parameters contain many zeros), and proceed with statistical inference based on this assumption. Lasso is a suitable algorithm for this setting since it conducts an implicit model selection, i.e., zeroing out parameters that are not significant, see Tibshirani \cite{10.2307/2346178}. More recent work includes Zhao and Yu \cite{10.5555/1248547.1248637}, Meinshausen and B\"{u}hlmann \cite{10.1214/009053606000000281} and Meinshausen and Yu \cite{10.1214/07-AOS582} for model selection; Zhang and Zhang \cite{10.2307/24772752}, Zhang and Cheng \cite{doi:10.1080/01621459.2016.1166114} and Chatterjee and Lahiri \cite{10.2307/41059185, doi:10.1198/jasa.2011.tm10159} for statistical inference and hypothesis testing; Greenshtein and Ritov \cite{10.3150/bj/1106314846} for prediction and Zou \cite{doi:10.1198/016214506000000735} for algorithm improvement. We refer B\"{u}hlmann and van de Geer \cite{statforhighdimen} for a comprehensive overview of the Lasso method on high dimensional data set.

Lasso is not the only choice for fitting a high-dimensional linear model. Fan and Li \cite{doi:10.1198/016214501753382273} introduced a new penalty function, called SCAD, that is continuously differentiable and maintains the sparsity of the underlying model. Lee et al. \cite{10.1214/15-AOS1371},  Liu and Yu \cite{liu2013} and Tibshirani et al. \cite{10.1214/17-AOS1584} considered \textit{Post-selection inference}, i.e., performing model selection with Lasso, then fitting ordinary least square regression on the selected parameters. Shao and Deng \cite{shao2012} applied thresholding on ridge regression to recover the sparsity of a linear model; Zhang and Politis \cite{Zhang2020RidgeRR} made a further improvement through debiasing and thresholding. After debiasing and thresholding, the ridge regression estimator had a comparable performance to the complex models like threshold Lasso or post-selection inference. Moreover, the estimator had a closed-form formula, making it easy to derive theoretical guarantees.

This paper focuses on fitting a high dimensional linear model, constructing confidence intervals, and performing hypothesis testing, with the presence of \textit{dependent and heteroskedastic} errors. Since non-i.i.d. errors may appear in fixed dimensional linear models, they may also appear in high dimensional linear models. Unfortunately, performing statistical inference for a high dimensional linear model with dependent and heteroskedastic errors is challenging. Wu and Wu \cite{10.1214/16-EJS1108} proposed an oracle inequality; Han and Tsay \cite{LassoDep} proved the consistency of the Lasso estimator. However, their works relied on a stationary assumption, e.g., definition 1.3.3 in Brockwell and Davis \cite{10.5555/17326}.

Suppose a high dimensional sparse linear model $y = X\beta + \epsilon$ with dependent, non-stationary, heteroskedastic errors $\epsilon = (\epsilon_1,...,\epsilon_n)^T$. Here $\epsilon_i, i = 1,...,n$ are not necessary to be linear processes. This paper focuses on estimating linear combinations of parameters $\gamma = M\beta$ with $M$ a given matrix. We adopt the debiased and threshold ridge regression estimator proposed in Zhang and Politis \cite{Zhang2020RidgeRR}. After selecting a suitable ridge parameter and a threshold, this method has a comparable performance to Lasso and is easily analyzed.

Our work includes proving the model selection consistency as well as the consistency of the estimator; deriving a Gaussian approximation theorem for the estimator and constructing a confidence interval for $\gamma$. We are also interested in testing the statistical hypothesis
\begin{equation}
\text{null: } M\beta = z\ \text{versus the alternative: } M\beta\neq z
\label{test}
\end{equation}
with $z$ a given vector. To achieve this goal, we adopt a dependent wild bootstrap(Shao \cite{doi:10.1198/jasa.2009.tm08744}) and provide its theoretical guarantee. Since there is little research on statistical inference for a high dimensional linear model with non-i.i.d. errors, this paper should shed some light on this field.

The remainder of this paper is organized as follows: section \ref{Prelim} introduces frequently used notations and assumptions. Section \ref{Consistency} presents the consistency results and the Gaussian approximation theorem for the proposed estimator. Section \ref{Bootstrap} constructs a confidence interval for $\gamma = M\beta$, and tests the null hypothesis $\gamma = z$ versus the alternative hypothesis $\gamma\neq z$ via dependent wild bootstrap. Section \ref{numericla} provides numerical experiments, and section \ref{conclusion} makes conclusions. Technical proofs are deferred to the Appendix.
\section{Preliminarily}
\label{Prelim}
Our work considers the fixed design linear model
\begin{equation}
y = X\beta + \epsilon
\label{linearModel}
\end{equation}
where the unknown parameter vector $\beta = (\beta_1,...,\beta_p)^T\in\mathbf{R}^p$, and the $n\times p (p<n)$ \textit{fixed(nonrandom)} design matrix $X = (x_{ij})_{i = 1,...,n,j=1,...,p}$ is assumed to have rank $p$. Denote $y = (y_1,...,y_n)^T$ and the errors $\epsilon = (\epsilon_1,...,\epsilon_n)^T$. From the thin singular value decomposition(theorem 7.3.2 in \cite{matrix}), $X = P\Lambda Q^T$ with $P = (p_{ij})_{i = 1,2,...,n, j =1,2,...,p},Q = (q_{ij})_{i,j = 1,2,...,p}$ respectively being the $n\times p, p\times p$ orthonormal matrix, i.e., $P^TP = Q^TQ = QQ^T = I_p$. Here $I_p$ denotes the $p\times p$ identity matrix. And $\Lambda = diag(\lambda_1,...,\lambda_p)$, $\lambda_i, i =1,2,...,p$ are singular values of $X$. Define the set $\mathcal{N}_b = \left\{i = 1,2,...,p: \vert\beta_i\vert>b\right\}, \forall b>0$. We are interested in constructing the confidence region for $\gamma = M\beta$, and testing the statistical hypothesis \eqref{test}.
Here $M = (m_{ij})_{i = 1,2,...,p_1,j = 1,2,...,p}$ is a given $p_1 \times p$ linear combination matrix, and $z = (z_1,...,z_{p_1})^T$ is the given expected value of $M\beta$.

After choosing a suitable threshold $b$, define $c_{ij}, i = 1,2,...,p_1, j = 1,2,...,p$ as $c_{ij} = \sum_{k\in\mathcal{N}_b}m_{ik}q_{kj}$, and $\mathcal{M} = \left\{i = 1,2,...,p_1: \sum_{j = 1}^p c^2_{ij} > 0\right\}$. Define

\noindent$\tau_{i} = \sqrt{\frac{1}{n} + \sum_{j = 1}^p c^2_{ij}\left(\frac{\lambda_j}{\lambda^2_j + \rho_n} + \frac{\rho_n\lambda_j}{(\lambda^2_j+\rho_n)^2}\right)^2}, i = 1,2,...,p_1$. Define $\Sigma = (\sigma_{ij})_{i,j = 1,2,...,n}$ with $\sigma_{ij} = \mathbf{E}\epsilon_i\epsilon_j$ as the covariance matrix of errors.

To quantify the dependency among random variables, Wu \cite{Wu14150}(also see Wu and Wu \cite{wu2016}) introduced the concept `physical dependence'. However, this concept was designed for stationary random variables. We extend this concept to adapt to non-stationary errors $\epsilon$: suppose $e_i, i =...,-1,0,1,...$ are independent(not necessarily identically distributed) random variables, and $\epsilon_i = g_i(...,e_{i-1}, e_i)$, here $g_i$ are measurable functions, $i =1,2,...,n$. Define $\mathcal{F}_i$ as the $\sigma$-field generated by $...,e_{i-1},e_i$, so $\epsilon_i$ is $\mathcal{F}_i$ measurable. Suppose independent random variables $e_i^\dagger, i = ...,-1,0,1,...$ are independent with $e_j, j = ...,-1,0,1,...$; and $e_i^\dagger$ has the same distribution as $e_i$, $\forall i$. Define $\epsilon_{i,j} = g_i(...,e_{i-j-2}, e_{i-j-1}, e_{i-j}^\dagger, e_{i-j+1},...,e_i)$; and the filter $\mathcal{F}_{i,j}$ as the $\sigma$-field generated by $e_{i-j}, e_{i-j+1},...,e_i$. Here $i\in\mathbf{Z}, j\geq 0$. For a constant $m\geq 1$, define the norm $\Vert\cdot\Vert_m = (\mathbf{E}\vert \cdot\vert^m)^{1/m}$. Then define $\delta_{i,j,m} = \Vert\epsilon_i - \epsilon_{i,j}\Vert_m$.

This paper applies the standard order notations $O(.), o(.), O_p(.)$ and $o_p(.)$: for two numerical sequence $a_n,b_n$, we say $a_n = O(b_n)$ if $\exists$ a constant $C>0$ such that $\vert a_n\vert\leq C\vert b_n\vert$ for all $n$; and $a_n = o(b_n)$ if $\lim_{n\to\infty} a_n/b_n = 0$. For two random variable sequences $X_n, Y_n$, we say $X_n = O_p(Y_n)$ if for any $0<\varepsilon<1$, $\exists $ a constant $C_\varepsilon$ such that $\sup_{n} Prob(\vert X_n\vert\geq C_\varepsilon \vert Y_n\vert)\leq \varepsilon$; and $X_n = o_p(Y_n)$ if $X_n/Y_n \to_p 0$. See definition 1.9 and chapter 1.5.1 of Shao \cite{Mstat}. \textit{All order notations and convergences in this paper are understood to hold as the sample size $n\to\infty$. }

For a set $A$, we use $\vert A\vert$ to indicate the number of elements in $A$. For a vector $a = (a_1,...,a_n)^T$, define its $q$ norm as $\vertiii{a}_q = (\sum_{i = 1}^n \vert a_i\vert^q)^{1/q}$. For a matrix $T$, define(and mainly use) the operator norm $\vertiii{T}_2  = \max_{\vertiii{a}_2 = 1}\vertiii{Ta}_2$. The notation $\to$ and $\to_p$ respectively indicates convergence in $\mathbf{R}$, and convergence in probability. $\exists$ and $\forall$ respectively means `there exists' and `for all'. Define $Prob^*(\cdot) = Prob(\cdot|y)$ and $\mathbf{E}^*\cdot = \mathbf{E}(\cdot|y)$, i.e., the probability and the expectation in the bootstrap world.

This paper uses the following assumptions:

\textbf{Assumptions}

1. The \textbf{fixed} design matrix $X$ has rank $p\leq n$. There exists constants $c_\lambda,\ C_\lambda>0, 1/2\geq \eta>0$ such that
\begin{equation}
C_\lambda n^{1/2}\geq \lambda_1\geq\lambda_2\geq...\geq \lambda_p\geq c_\lambda n^\eta, \forall n
\end{equation}
Besides, $\max_{i =1,...,n, j = 1,...,p}\vert x_{ij}\vert = O(1)$. $p_1 = O(1)$. Here $p_1$ is the number of simultaneous linear combinations in \eqref{test}.

2. $\mathbf{E}\epsilon_i = 0, i = 1,2,...,n$. There exists constants $m > \frac{3}{\eta}$ and $\alpha_\epsilon > 1$ such that
\begin{equation}
\sup_{k = 0, 1,...}(k+1)^{\alpha_\epsilon}\sum_{j = k}^\infty \max_{i = 1,2,...,n}\delta_{i,j,m} = O(1)\ \text{and }  \max_{i = 1,2,...,n}\Vert\epsilon_i\Vert_m = O(1)
\label{cond1}
\end{equation}
Besides, $\exists$  a constant $c_\Sigma > 0$ such that $\Sigma$'s minimum eigenvalue is greater than $c_\Sigma$ for $\forall n$; here $\Sigma = \mathbf{E}\epsilon\epsilon^T$.

3. $\vertiii{\beta}_2 = O(n^{\alpha_\beta})$ with a constant $\alpha_\beta$ such that $0\leq \alpha_\beta < 3\eta$, and $\vert\mathcal{N}_{b_n}\vert = o(n^{\eta-1/m})$.

4. $\rho_n = O(n^{2\eta-\delta})$ with a constant $\delta$ such that $\frac{\eta + \alpha_\beta}{2}<\delta<2\eta$. Besides, there exists constants $C_b > 0$ and $0<\nu_b <\eta -  \frac{3}{m}$ such that $b_n = C_b n^{-\nu_b}$ for $\forall n$. Assume $\exists$ a constant $0<c_b<1$ such that $\max_{i\not\in\mathcal{N}_{b_n}}\vert\beta_i\vert\leq c_b\times b_n$ and $\min_{i\in\mathcal{N}_{b_n}}\vert\beta_i\vert\geq b_n/c_b$.

5. $\mathcal{M}$ is not empty. And $\exists$ constants $0<c_\mathcal{M}\leq C_\mathcal{M}<\infty$ such that $c_{\mathcal{M}}\leq\sum_{k\in\mathcal{N}_{b_n}}m^2_{ik}(=\sum_{j = 1}^pc^2_{ij})\leq C_\mathcal{M} $ for $\forall i\in\mathcal{M}$ and $\forall n$. Also assume

\begin{equation}
\max_{i = 1,...,p_1}\vert\sum_{j\not\in\mathcal{N}_{b_n}}m_{ij}\beta_j\vert = o(1/\sqrt{n}) \text{ and } \max_{i = 1,...,n}\vert\sum_{j\not\in\mathcal{N}_{b_n}}x_{ij}\beta_j\vert = o(n^{\frac{1/m -\eta}{2}})
\end{equation}

6. The matrix $(c_{ij})_{i\in\mathcal{M}, j = 1,2,...,p}$ has rank $\vert\mathcal{M}\vert$ and
\begin{equation}
c^* = \max_{i\in\mathcal{M}, l = 1,2,...,n}\frac{1}{\tau_i}\vert\sum_{j = 1}^p c_{ij}p_{lj}\left(\frac{\lambda_j}{\lambda^2_j + \rho_n} + \frac{\rho_n\lambda_j}{(\lambda^2_j + \rho_n)^2}\right)\vert = o(n^{-1/4}\log^{-z}(n))
\end{equation}
Here $z = \max(\frac{9}{2}, \frac{3\alpha_\epsilon}{2\alpha_\epsilon-2})$.

7. Suppose a function $K:\mathbf{R}\to[0,\infty)$ is symmetric, continuously differentiable, $K(0) = 1$, $\int_{\mathbf{R}} K(x)dx <\infty$, and $K$ is decreasing on $[0,\infty)$. Define the Fourier transformation of $K$ as $\mathcal{F}K(x) = \int_{\mathbf{R}}K(t)\exp(-2\pi \rm{i}tx)dt$. Assume $\mathcal{F}K(x)\geq 0, \forall x\in\mathbf{R}$ and $\int_{\mathbf{R}}\mathcal{F}K(x)dx<\infty$. Define a bandwidth parameter $k_n$ satisfying $\lim_{n\to\infty} k_n = \infty$, $k_n = o(n^{2\eta-2\nu_b-4/m})$, and $k_n = O(n^{\frac{\eta} {2}-\frac {3}{2m}})$.

In assumption 6, $(c_{ij})_{i\in\mathcal{M}, j = 1,2,...,p}$ denotes the sub-matrix formed by selecting  $\mathcal{M}$ rows from the matrix $(c_{ij})_{i = 1,...,p_1, j = 1,...,p}$.

Our work mainly focuses on fixed design, i.e., no randomness involves in the design matrix $X$. In the case of random design, results in this paper hold true after conditioning on the design matrix $X$, i.e., replacing $Prob(\cdot)$ by $Prob(\cdot| X)$;  $\mathbf{E}\cdot$ by $\mathbf{E}\cdot| X$; $Prob^*(\cdot)$ by $Prob(\cdot|X,y)$ and $\mathbf{E}^*\cdot$ by $\mathbf{E}\cdot|X,y$.

Example \ref{examp1} introduces a situation in which assumption 1 is satisfied.

\begin{example}
\label{examp1}
Suppose $n > p$ and $p / n \to c \in (0, 1)$. Choose $X = (x_{ij})_{i = 1,...,n, j = 1,...,p}$ such that each $x_{ij}$ is a realization of i.i.d. random variables with mean $0$, variance $1$, and finite fourth order moment. According to Bai and Yin \cite{10.1214/aop/1176989118}, the smallest eigenvalue of $\frac{1}{n}X^TX$ converges to $(1 - \sqrt{c})^2$ almost surely as $n\to\infty$. So the smallest singular value of $X$, being the square root of $X^TX$'s smallest eigenvalue, is greater than $\frac{1 - \sqrt{c}}{2}\sqrt{n}$ for sufficiently large $n$ almost surely. On the other hand, the largest eigenvalue of $\frac{1}{n}X^TX$ converges to $(1 + \sqrt{c})^2$ as $n\to\infty$. So the largest singular value of $X$ has order $O(\sqrt{n})$ almost surely.
\end{example}

\begin{remark}
1. Under stationary assumptions, e.g., Wu \cite{Wu14150}
\begin{equation}
\delta_{i,j,m} = \Vert\epsilon_0 - \epsilon_{0,j}\Vert_m = \delta_{0,j,m}\Rightarrow \sum_{j = k}^\infty \max_{i = 1,...,n}\delta_{i,j,m} = \sum_{j = k}^\infty\delta_{0,j,m}
\end{equation}
and \eqref{cond1} coincides with (2.8) in Wu and Wu \cite{wu2016}, i.e., the dependence adjusted norm condition. Therefore, assumption 2 can be recognized as the dependence adjusted norm condition for non-stationary random variables.

\noindent 2. According to Shao \cite{doi:10.1198/jasa.2009.tm08744} and Fourier inversion theorem(theorem 8.26 in \cite{RealAnalysis}), $\forall x = (x_1,...,x_n)^T\in\mathbf{R}^n$,
\begin{equation}
\begin{aligned}
\sum_{s = 1}^n\sum_{j = 1}^n x_s x_j K\left(\frac{s - j}{k_n}\right) =  \int_{\mathbf{R}}\sum_{s = 1}^n\sum_{j = 1}^n x_s x_j\mathcal{F}K(z)\exp\left(2\pi \rm{i}z\frac{s-j}{k_n}\right)dz\\
= \int_{\mathbf{R}}FK(z)\vert\sum_{s = 1}^n x_s\exp(\frac{2\pi \rm{i} zs}{k_n})\vert^2dz\geq 0
\end{aligned}
\end{equation}
Therefore, the matrix $\left\{K\left(\frac{s - j}{k_n}\right)\right\}_{s,j = 1,2,...,n}$ is symmetric positive semi-definite. One example of $K$ satisfying assumption 7 is $K(x) = \exp(-x^2/2)\Rightarrow \mathcal{F}K(x)  = \sqrt{2\pi}\exp(-2\pi^2x^2)$.

\noindent 3. Since
\begin{equation}
\sum_{l = 1}^n \left(\sum_{j = 1}^pc_{ij}p_{lj}\left(\frac{\lambda_j}{\lambda_j^2 + \rho_n} + \frac{\rho_n\lambda_j}{(\lambda^2_j + \rho_n)^2}\right)\right)^2 = \sum_{j = 1}^p c^2_{ij}\left(\frac{\lambda_j}{\lambda^2_j + \rho_n} + \frac{\rho_n\lambda_j}{(\lambda^2_j + \rho_n)^2}\right)^2 < \tau_i
\end{equation}
Assumption 6 requires no single element in the array $\left\{\sum_{j = 1}^pc_{ij}p_{lj}\left(\frac{\lambda_j}{\lambda_j^2 + \rho_n} + \frac{\rho_n\lambda_j}{(\lambda^2_j + \rho_n)^2}\right)\right\}_{l = 1,...,n}$ dominates others. We add $1/n$ in $\tau_i$ to prevent the normalizing parameters from being $0$.

\noindent 4. Assumption 2 implies a polynomial decay of $\epsilon$'s covariance, i.e., $\max_{1\leq i < j\leq n}(1 + j-i)^{\alpha_\epsilon}\times \vert\mathbf{E}\epsilon_i\epsilon_j\vert = O(1)$. We will prove this in appendix \ref{UsefulRes}. Panel data may require different types of dependency, e.g., clustered dependency or user-defined spatial dependency(Vogelsang \cite{VOGELSANG2012303}). Our work needs to be adjusted if those dependencies show up.
\end{remark}
\section{Consistency and Gaussian approximation theorem}
\label{Consistency}

This paper applies the ridge regression estimator introduced in Zhang and Politis \cite{Zhang2020RidgeRR}. Use notations in section \ref{Prelim}. Suppose the classical ridge regression estimator $\widetilde{\beta}^\star$ and the de-biased estimator $\widetilde{\beta}$ as
\begin{equation}
\begin{aligned}
\widetilde{\beta}^\star = (X^TX + \rho_n I_p)^{-1}X^Ty\\
\widetilde{\beta} = \widetilde{\beta}^\star + \rho_nQ(\Lambda^2 + \rho_n I_p)^{-1}Q^T\widetilde{\beta}^\star
\end{aligned}
\end{equation}
Define $\widehat{\mathcal{N}}_{b_n} = \left\{i = 1,2,...,p: \vert\widetilde{\beta}_i\vert > b_n\right\}$, here $b_n>0$ is a given number(related to the sample size $n$). Define $\widehat{\beta} = (\widehat{\beta}_1,...,\widehat{\beta}_p)^T$ such that $\widehat{\beta}_i = \widetilde{\beta}_i \times \mathbf{1}_{i\in\widehat{\mathcal{N}}_{b_n}}, i = 1,2,...,p$; and $\widehat{\gamma} = (\widehat{\gamma}_1,...,\widehat{\gamma}_{p_1})^T=M\widehat{\beta}$. $\mathbf{1}_{i\in\widehat{\mathcal{N}}_{b_n}} = 1$ if $i\in\widehat{\mathcal{N}}_{b_n}$ and $0$ otherwise.  This paper will use $\widehat{\gamma}$ to estimate $\gamma = (\gamma_1,...\gamma_{p_1})^T= M\beta$.

\begin{remark}
Remark 2 in Zhang and Politis \cite{Zhang2020RidgeRR} explained why $\widetilde{\beta}$ helps decrease the bias as well as the  estimation error. Generally speaking, the bias of a classical ridge regression estimator $\widetilde{\beta}^\star$ is $-\rho_n Q(\Lambda^2 + \rho_n I_p)^{-1}Q^T\beta$. We can estimate the bias by $-\rho_n Q(\Lambda^2 + \rho_n I_p)^{-1}Q^T\widetilde{\beta}^\star$, then subtract the estimated bias from $\widetilde{\beta}^\star$, yielding the new estimator $\widetilde{\beta}$. Compared to $\widetilde{\beta}^\star$, $\widetilde{\beta}$ increases the variance but decreases the bias. So with a suitable choice of $\rho_n$, the estimation error will decrease.

The consistency of $\widetilde{\beta}$ is not a free lunch. We need $\vertiii{\beta}_2$ not to be very large(which is achieved if $\beta$ is sparse) and $\rho_n/\lambda_p^2\to 0$ to maintain consistency.
\end{remark}

The consistency of $\widetilde{\beta}$ consists of two aspects: the model selection consistency, i.e., $\widehat{\mathcal{N}}_{b_n} = \mathcal{N}_{b_n}$(defined in section \ref{Prelim}) with probability tending to $1$. And the consistency of $\widehat{\gamma}$, i.e., $\max_{i = 1,...,p_1}\vert\widehat{\gamma}_i - \gamma_i\vert = o_p(1)$. Theorem \ref{thm1} will prove both consistencies. Another important result in theorem \ref{thm1} is that the estimated errors $\widehat{\epsilon} = y - X\widehat{\beta}$ are close to the underlying errors $\epsilon$. This result is the theoretical foundation for bootstrap algorithm \ref{alg1}.

\begin{theorem}
Suppose assumption 1 to 5. Then

1.
\begin{equation}
Prob\left(\widehat{\mathcal{N}}_{b_n}\neq \mathcal{N}_{b_n}\right) = o(1)\ \text{and }\max_{i = 1,2,...,p_1}\vert\widehat{\gamma}_i - \gamma_i\vert = O_p(n^{-\eta})
\end{equation}

2. Define $\widehat{\epsilon} = (\widehat{\epsilon}_1,...,\widehat{\epsilon}_n)^T$ such that $\widehat{\epsilon}_i = y_i - \sum_{j = 1}^px_{ij}\widehat{\beta}_j$, then
\begin{equation}
\max_{i = 1,2,...,n}\vert\widehat{\epsilon}_i - \epsilon_i\vert = o_p(n^{1/(2m) - \eta/2})
\label{ConsHalf}
\end{equation}
\label{thm1}
\end{theorem}
\eqref{ConsHalf} is not self-evident for a high dimensional linear model. We refer Mammen \cite{10.1214/aos/1033066211} for a detailed explanation.

If the dimension of a parameter vector changes according to the sample size, then the estimator $\widehat{\beta}$ does not have an asymptotic distribution, and the classical central limit theorem fails. Nevertheless, statisticians still can use normal random variables to approximate $\widehat{\beta}$'s distribution. In reality, joint normal random variables can be generated easily by a computer. So statisticians can perform Monte-Carlo simulations to derive the asymptotic probability of an event or make a consistent confidence interval. This idea is common for high dimensional statistics, see e.g., Chernozhukov et al. \cite{chernozhukov2013} and Zhang and Wu \cite{10.1214/16-AOS1512}.

Following this idea, theorem \ref{Central} presents a Gaussian approximation theorem of $\widehat{\gamma}$. Define $\xi = (\xi_1,...,\xi_n)^T$ as joint normal random variables with $\mathbf{E}\xi = 0, \mathbf{E}\xi\xi^T = \Sigma$(the covariance matrix of errors $\epsilon$), and
\begin{equation}
H(x) = Prob\left(\max_{i\in\mathcal{M}}\frac{1}{\tau_i}\vert\sum_{j = 1}^p\sum_{l = 1}^nc_{ij}\left(\frac{\lambda_j}{\lambda^2_j + \rho_n} + \frac{\rho_n\lambda_j}{(\lambda^2_j + \rho_n)^2}\right)p_{lj}\xi_l\vert\leq x\right),\ x\in\mathbf{R}
\label{DefH}
\end{equation}
See section \ref{Prelim} for the meaning of notations. $\tau_i$ and $\mathcal{M}$ are defined by choosing $b = b_n$. According to assumption 5, $\mathcal{M}$ is not empty. So $H$ is well-defined.
\begin{theorem}
Suppose assumption 1 to 6. Then
\begin{equation}
\sup_{x\in\mathbf{R}}\vert Prob\left(\max_{i = 1,...,p_1}\frac{\vert\widehat{\gamma}_i - \sum_{j = 1}^p m_{ij}\beta_j\vert}{\widehat{\tau}_i}\leq x\right) - H(x)\vert = o(1)
\label{Cent}
\end{equation}
Here $\widehat{\tau}_i = \sqrt{\frac{1}{n} + \sum_{j = 1}^p \widehat{c}^2_{ij}\left(\frac{\lambda_j}{\lambda^2_j + \rho_n} + \frac{\rho_n\lambda_j}{(\lambda^2_j+\rho_n)^2}\right)^2}$ and $\widehat{c}_{ij} = \sum_{k\in\widehat{\mathcal{N}}_{b_n}}m_{ik}q_{kj}$; $i = 1,2,...,p_1, j = 1,2,...,p$.
\label{Central}
\end{theorem}

Define $c_{1-\alpha}$ as the $1-\alpha$ quantile of $H(x)$, i.e., $c_{1-\alpha} = \inf\left\{x\in\mathbf{R}: H(x)\geq 1-\alpha\right\}$. Theorem \ref{Central} implies that the set
\begin{equation}
\left\{\gamma\in\mathbf{R}^{p_1}:\max_{i = 1,...,p_1}\frac{\vert\widehat{\gamma}_i - \gamma_i\vert}{\widehat{\tau}_i}\leq c_{1-\alpha}\right\}
\label{conf}
\end{equation}
is an asymptotic consistent confidence region for $\gamma$.

\section{Dependent wild bootstrap}
\label{Bootstrap}
In order to calculate \eqref{conf}, we need a method that derives(or approximates) $c_{1-\alpha}$. Since $\mathcal{M}, \mathcal{N}_{b_n}$, and $\Sigma$ are unknown a priori and $H$ does not have a closed-form formula, we develop a Monte-Carlo algorithm(algorithm \ref{alg1}) that finds $c_{1-\alpha}$. The advantage of algorithm \ref{alg1} is that it avoids explicitly estimating $\Sigma$, which is hard for non-stationary random variables $\epsilon$. Our method applies the idea of \textit{the dependent wild bootstrap}. It was first introduced by Shao \cite{doi:10.1198/jasa.2009.tm08744}, and has been used in linear regression (Conley et al. \cite{Zhang}), time series analysis (Fragkeskou and Paparoditis \cite{https://doi.org/10.1111/jtsa.12275}, Shao \cite{doi:10.1198/jasa.2009.tm08744} and Doukhan et al. \cite{https://doi.org/10.1111/jtsa.12106}), high dimensional data analysis (Kurisu et al. \cite{shao}), etc.

Our target is to construct the confidence interval for $\gamma = M\beta$. And Perform the statistical hypothesis testing for the null hypothesis $M\beta = z$ with $z = (z_1,...,z_{p_1})^T$ a known vector versus the alternative hypothesis $M\beta\neq z$. Notably, this test problem receives lots of attentions in econometrics, e.g., Gon\c{c}alves and Vogelsang \cite{goncalves_vogelsang_2011}, Sun and Kim \cite{Suns}, and Conley et al. \cite{Conley}.

\begin{algorithm}[Bootstrap inference and hypothesis testing]
\textbf{Input: }Design matrix $X$ and dependent variables $y$ satisfying $y=X\beta + \epsilon$; the linear combination matrix $M$. Ridge parameter $\rho_n$, threshold $b_n$, bandwidth $k_n$, a kernel function $K:\mathbf{R}\to[0,\infty)$. Nominal coverage probability $1-\alpha$, number of bootstrap replicates $B$.

\noindent\textbf{Additional input for testing: }$z = (z_1,...,z_{p_1})^T$.

1. Derive the estimator $\widehat{\beta}$ defined in section \ref{Consistency} and $\widehat{\tau}_i, i = 1,...,n$ defined in theorem \ref{Central}. Then calculate $\widehat{\gamma} = (\widehat{\gamma}_1,...,\widehat{\gamma}_{p_1})^T = M\widehat{\beta},\ \widehat{\epsilon} = (\widehat{\epsilon}_1,...,\widehat{\epsilon}_n)^T = y-X\widehat{\beta}$.

2. Generate joint normal random variables $\varepsilon^*_1,...,\varepsilon^*_n$ with mean $0$ and covariance matrix $\left\{K\left(\frac{i-j}{k_n}\right)\right\}_{i,j=1,...,n}$. Then define $\epsilon^* = (\epsilon^*_1,...,\epsilon^*_n)^T$ such that $\epsilon^*_i = \widehat{\epsilon}_i\times \varepsilon^*_i, i = 1,2,...,n$.

3. Calculate
\begin{equation}
\begin{aligned}
y^* = X\widehat{\beta} + \epsilon^*\\
\widetilde{\beta}^{\star*} = (X^TX + \rho_n I_p)^{-1}X^Ty^*\\
\widetilde{\beta}^*= (\widetilde{\beta}^*_1,...,\widetilde{\beta}^*_p)^T  = \widetilde{\beta}^{\star*} + \rho_nQ(\Lambda^2 + \rho_n I_p)^{-1}Q^T\widetilde{\beta}^{\star*}\\
\widehat{\mathcal{N}}_{b_n}^* = \left\{i=1,2,...,p: \vert\widetilde{\beta}_i^*\vert>b_n\right\}\\
\text{and }\widehat{\beta}^* = (\widehat{\beta}_1^*,...,\widehat{\beta}^*_p)\ \text{such that }\widehat{\beta}^*_i = \widetilde{\beta}^*_i\times \mathbf{1}_{i\in\widehat{\mathcal{N}}_{b_n}^*}
\end{aligned}
\label{calBoot}
\end{equation}

4. Derive $\widehat{\gamma}^* = (\widehat{\gamma}^*_1,...,\widehat{\gamma}^*_{p_1})^T = M\widehat{\beta}^*$,   $\widehat{c}_{ij}^* = \sum_{k\in\widehat{\mathcal{N}}_{b_n}^*}m_{ik}q_{kj}$ and

\noindent $\widehat{\tau}_{i}^* = \sqrt{\frac{1}{n} + \sum_{j = 1}^p \widehat{c}^{*2}_{ij}\left(\frac{\lambda_j}{\lambda^2_j + \rho_n} + \frac{\rho_n\lambda_j}{(\lambda^2_j+\rho_n)^2}\right)^2}$. Then calculate
\begin{equation}
\delta^*_b = \max_{i = 1,2,...,p_1}\frac{\vert\widehat{\gamma}^*_i - \widehat{\gamma}_i\vert}{\widehat{\tau}^*_i}
\end{equation}

5. Repeat step 2 to 4 for $b = 1,...,B$, and calculate the $1-\alpha$ sample quantile $\widehat{c}^*_{1-\alpha}$ of $\delta^*_b, b =1,2,...,B$.

6.a(For constructing the confidence region) The $1-\alpha$ confidence region for the parameter of interest $\gamma =M\beta$ is given by the set
\begin{equation}
\left\{\gamma = (\gamma_1,...,\gamma_{p_1})^T: \max_{i = 1,...,p_1}\frac{\vert\widehat{\gamma}_i - \gamma_i\vert}{\widehat{\tau}_i}\leq \widehat{c}^*_{1-\alpha}\right\}
\end{equation}

6.b(For hypothesis testing) Reject the null hypothesis if
\begin{equation}
\max_{i =1,...,p_1}\frac{\vert\widehat{\gamma}_i - z_i\vert}{\widehat{\tau}_i} > \widehat{c}^*_{1-\alpha}
\end{equation}
\label{alg1}
\end{algorithm}
\textit{For a numerical sequence $\delta_1\leq\delta_2\leq...\leq\delta_n$, the $1 - \alpha$ sample quantile $C_{1-\alpha}$ is defined as}
\begin{equation}
C_{1-\alpha} = \delta_{i*} \text{  such that  } i* = \min\left\{i = 1,...,n: \frac{1}{n}\sum_{j = 1}^n \mathbf{1}_{\delta_j\leq \delta_{i}}\geq 1- \alpha\right\}
\end{equation}

\noindent Define the conditional quantile of $\delta^*_b, b = 1,...,B$ as
\begin{equation}
c^*_{1-\alpha} = \inf\left\{x\in\mathbf{R}: Prob^*\left(\max_{i = 1,...,p_1}\frac{\vert\widehat{\gamma}^*_i-\widehat{\gamma}_i\vert}{\widehat{\tau}^*_i}\leq x\right)\geq 1-\alpha\right\},\ 0<\alpha<1
\end{equation}
According to theorem 1.2.1 in Politis et al. \cite{subsampling}, the sample quantile $\widehat{c}^*_{1-\alpha}$ converges to $c^*_{1-\alpha}$ as $B\to\infty$. So consistency of algorithm \ref{alg1} is assured by showing

\noindent$Prob\left(\max_{i =1,...,p_1}\frac{\vert\widehat{\gamma}_i - \gamma_i\vert}{\widehat{\tau}_i} \leq c^*_{1-\alpha}\right)\to 1-\alpha$. We derive this result in theorem \ref{thmBootS}.

\begin{theorem}
Suppose assumption 1 to 7. Then
\begin{equation}
\begin{aligned}
\sup_{x\in\mathbf{R}}\vert Prob^*\left(\max_{i = 1,...,p_1}\frac{\vert\widehat{\gamma}^*_i-\widehat{\gamma}_i\vert}{\widehat{\tau}^*_i}\leq x\right) - H(x)\vert = o_p(1)\\
\text{and }Prob\left(\max_{i =1,...,p_1}\frac{\vert\widehat{\gamma}_i -\gamma_i\vert}{\widehat{\tau}_i} \leq c^*_{1-\alpha}\right)\to 1-\alpha
\end{aligned}
\label{thmBoot}
\end{equation}
Here $H$ is defined in \eqref{DefH}.
\label{thmBootS}
\end{theorem}

\section{Numerical simulations}
\label{numericla}
\subsection{Selecting hyper-parameters}
\label{finetune}
Statisticians need to fine-tune the ridge regression parameter $\rho_n$, the threshold $b_n$, and the bandwidth $k_n$. $\rho_n$ and $b_n$ can be chosen by (ten-fold) cross-validation, i.e., separate the design matrix $X$ and the dependent variables $y$ into disjoint training set $(X_{train}, y_{train})$ and validation set $(X_{valid}, y_{valid})$; for each choice of $\rho_n$ and $b_n$, use $(X_{train}, y_{train})$ to fit $\widehat{\beta}$; and calculate $\vertiii{y_{valid} - X_{valid}\widehat{\beta}}_2$. The optimal $\rho_n$ and $b_n$ should minimize $\vertiii{y_{valid} - X_{valid}\widehat{\beta}}_2$. This paper implements the ten-fold cross validation. See Arlot and Celisse \cite{10.1214/09-SS054} for a further introduction on the cross validation methods.

Fine-tuning $k_n$ is more challenging. Politis and White \cite{doi:10.1081/ETC-120028836} introduced an automatic bandwidth selection algorithm; Shao \cite{doi:10.1198/jasa.2009.tm08744} applied this algorithm for the dependent wild bootstrap. Andrews \cite{10.2307/2938229} and Kim and Sun \cite{KIM2011349} considered selecting bandwidth in the HAC estimation setting. Following Shao, this paper applies Politis and White's algorithm \cite{doi:10.1081/ETC-120028836} on fitted residuals $\widehat{\epsilon} = (\widehat{\epsilon}_1,...,\widehat{\epsilon}_n)^T = y - X\widehat{\beta}$ to select $k_n$. However, $\epsilon$ is not assumed to be stationary. So this algorithm may result in a suboptimal bandwidth.

\subsection{Generating data}
\label{genData}
The numerical experiment applies the linear model $y = X\beta + \epsilon$. $X$ is generated by i.i.d. normal random variables with mean $0$ and variance $1$, and is fixed in each experiment. $\beta = (\beta_1,...,\beta_p)^T$ is generated by the following strategy
\begin{equation}
\beta_i = 0.1\times(i + 6)\ \text{for } i = 1,2,...,10\ \text{and } 0\ \text{otherwise}
\end{equation}
Define $e_{i}, i \in\mathbf{Z}$ as i.i.d. normal random variables with mean $0$ and variance $1$. Choose $a = (a_1,a_2,...,a_n)^T$ such that $a_i = e^2_ie^2_{i - 1} - 1$. Then define $\epsilon = (\epsilon_1,...,\epsilon_n)^T = (1/4)\times Ha$(the factor $1/4$ avoids $\epsilon$'s variances being too large). Here $H = (h_{ij})_{i ,j = 1,...,n}$, $h_{ij} = 0$ for $j > i$, $h_{ii} = 1$, $h_{ij}$ is generated by uniform distribution in $[0.6, 0.9]$ for $i-10\leq j<i$ and $h_{ij} = s_{ij} / (i - j)^z$ for $j < i - 10$. $s_{ij}$ is generated by uniform distribution on $[-1, 1]$. $H$ is fixed in each experiment. For $\mathbf{E}a_i = \mathbf{E}e^2_i\times \mathbf{E}e^2_{i - 1} - 1 = 0$, we have $\mathbf{E}\epsilon = 0$. Moreover, $a$ is not white noise since $\mathbf{E}a_ia_{i-1} = 2$. $\epsilon$ is not linear because of $a$, and is not stationary because of $H$. Figure \ref{datalike} plots an observation of the errors $\epsilon$, and the first $10 \times 10$ elements of $\epsilon$'s covariance matrix. In figure \ref{onea}, the errors demonstrate strong dependency, i.e., $\epsilon_{i+1}$ is likely to be large if $\epsilon_i$ is large. In figure \ref{oneb}, $\epsilon$'s covariance matrix is not a Toeplitz matrix, so the distribution of $\epsilon$ is not stationary.

\begin{figure}[htbp]
\centering
    \subfigure[]{
    \includegraphics[width = 2.5in]{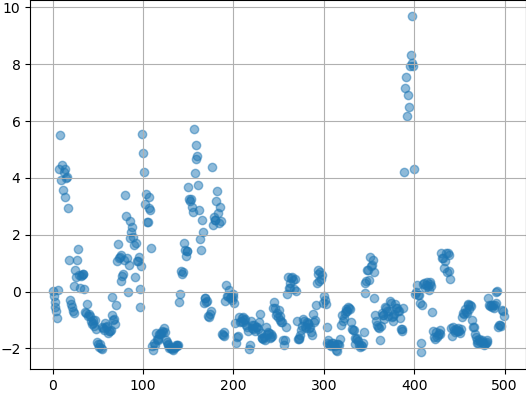}
    \label{onea}
  }
    \subfigure[]{
    \includegraphics[width = 3.0in]{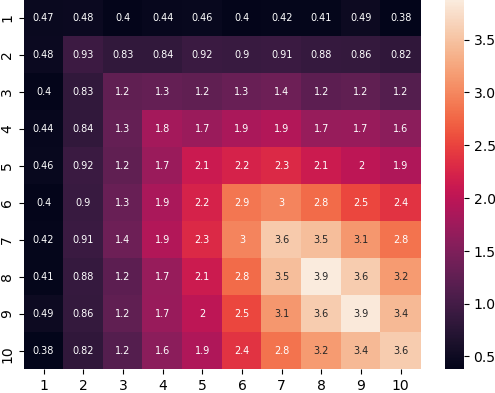}
    \label{oneb}
  }
  \caption{Figure \ref{onea} plots an observation of the errors $\epsilon$, and figure \ref{oneb} plots the heatmap for the first $10\times 10$ elements of $\epsilon$'s covariance matrix. Values in each grid represent the corresponding covariance. The covariance matrix is calculated by simulating 40000 samples of the random vector $\epsilon = (\epsilon_1,...,\epsilon_n)^T$.}
    \label{datalike}
\end{figure}

The linear combination matrix $M$ is generated by i.i.d. normal random variables with mean 0.5 and variance $0.25$, and is fixed in each experiment. The hyper-parameters $\rho_n, b_n, k_n$ are tuned by methods described in section \ref{finetune}. The sample size $n$, dimension $p$ and $z$ vary in each experiment. We store the information about experiments in table \ref{ExpRes}.

\begin{table}[htbp]
  \centering
  \caption{Experiment parameters. `No.' abbreviates `the experiment number', $n$ is the sample size. $p$ is the dimension of parameters. $p_1$ is the number of linear combinations. $z$ is defined in section \ref{genData}. $\rho_n,b_n,k_n$ respectively is the selected ridge parameter, the selected threshold, and the selected bandwidth defined in section \ref{Prelim}. $\lambda_p$ is the smallest singular value of the design matrix $X$. We use R-Package `np'\cite{Rnp} to facilitate choosing $k_n$.}
  \begin{tabular}{l l l l l l l l l}
  \hline\hline
  No. & $n$ & $p$ & $p_1$ & $z$ & $\rho_n$ & $b_n$ & $k_n$ & $\lambda_p$\\
  1 & 500   & 250 & 20 & 2.5 & 23.87 & 0.56 &  17.67 & 6.95\\
  2 & 500   & 400 & 20 & 2.5 & 34.97 & 0.66 &  23.42 & 2.47 \\
  3 & 500   & 250 & 60 & 2.5 & 69.22 & 0.47 &  16.75 & 6.66\\
  4 & 800   & 640 & 60 & 2.5 & 31.30 & 0.61 &  17.99 & 3.32 \\
  5 & 3000 & 1500 & 20  & 2.5 & 117.10 & 0.31 & 46.80 & 13.46\\
  6 & 500   & 250   &20 & 2.2 & 79.07  & 0.48 & 15.76  &6.53 \\
  \hline\hline
  \end{tabular}
  \label{ExpRes}
\end{table}

\subsection{Performance of linear regression algorithms}
This section compares the performance of the debiased and threshold ridge regression(thsDeb) to the classical linear regression algorithms including Lasso, the Ridge regression(Ridge), the threshold Lasso(thsLas), the threshold ridge regression(thsRid) and the elastic net. The evaluation indices consist of the estimation loss$\vertiii{M\widehat{\beta} - M\beta}_2$ and the prediction loss $\vertiii{X\widehat{\beta} - X\beta}_2$. The latter one is frequently used in evaluating linear regression algorithms, see e.g. Dalalyan et al. \cite{10.3150/15-BEJ756}.

We are also interested in the model selection performance(i.e., whether an algorithm can recover the sparsity of the underlying linear model or not) of the linear regression algorithms. Following the idea of Fithian et al. \cite{fithian}, we evaluate this through the frequency of model misspecification(i.e., $\widehat{\mathcal{N}}_{b_n}\neq \mathcal{N}_{b_n}$), $P(\widehat{\mathcal{N}}_{b_n}\neq \mathcal{N}_{b_n})$; the average size of model misspecification $\vert\widehat{\mathcal{N}}_{b_n}\Delta\mathcal{N}_{b_n}\vert$ ($\Delta$ denotes the symmetric difference, i.e. $A \Delta B = (A - B)\cup (B - A)$); and the average false discovery rate $\vert \widehat{\mathcal{N}}_{b_n} -  \mathcal{N}_{b_n}\vert / \max(1, \vert \widehat{\mathcal{N}}_{b_n} \vert)$. Since Lasso, ridge regression and elastic net do not have thresholds, for these algorithms we consider $i\in\widehat{\mathcal{N}}_{b_n}$ if $\vert\widehat{\beta}_i\vert > 0.01$.

According to figure \ref{Figure1}, figure \ref{figX} and table \ref{tab2}, the debiased and threshold ridge regression, the threshold Lasso and the threshold ridge regression have small  estimation loss and prediction loss compared to the other methods. Besides, these three methods are able to recover the underlying sparsity of the linear model correctly. Notably, even if Lasso is famous for generating sparse linear regression estimators, further thresholding still significantly improves its performance.

The necessity of debiasing can be illustrated in figure \ref{Figure1} and figure \ref{figX}. Although the threshold ridge regression has a small estimation loss when $\rho_n$ is close to its optimal value, the estimation loss surges when $\rho_n$ deviates from its optimal value. However, after debiasing, even if a cross validation selects a sub-optimal $\rho_n$, the estimator's performance does not notably deteriorate. In reality, cross validation algorithms cannot assure selecting the optimal ridge parameter $\rho_n$. So an algorithm that is not sensitive to the slight fluctuations in hyper-parameters is preferable.

Figure \ref{Figure1} also plots the relation between thresholds and the estimation losses. In the experiments, the cross validation always selects sub-optimal thresholds. However, the sub-optimal thresholds form wide-intervals. Therefore, slight fluctuations in the threshold $b_n$ should not deteriorate the performance of the regression algorithm as well.

\begin{figure}[htbp]
    \centering
    \subfigure[Experiment 1]{
    \includegraphics[width = 2.3in]{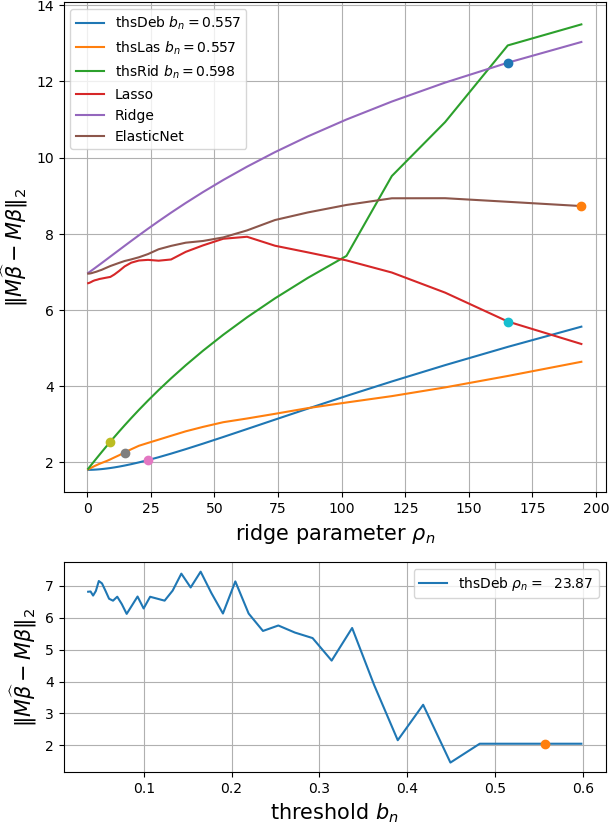}
  }
    \subfigure[Experiment 2]{
    \includegraphics[width = 2.3in]{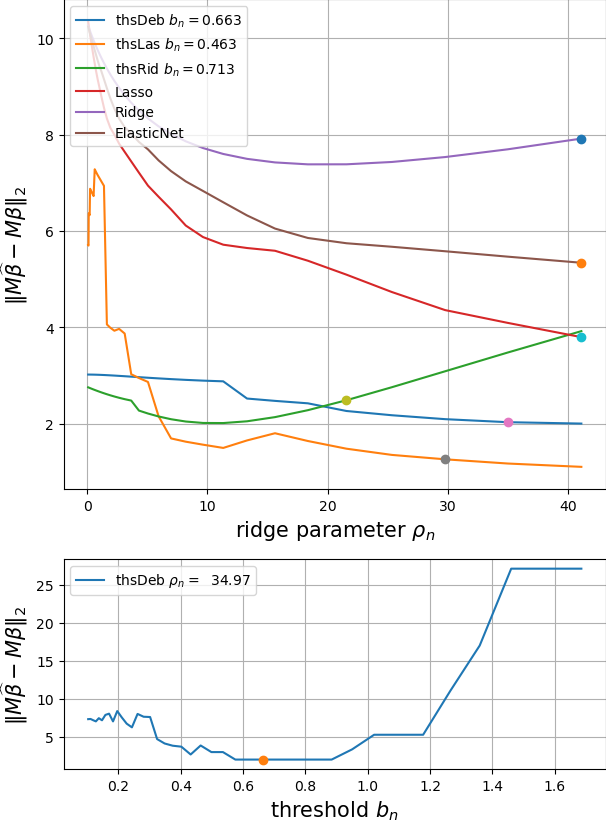}
  }
    \subfigure[Experiment 3]{
    \includegraphics[width = 2.3in]{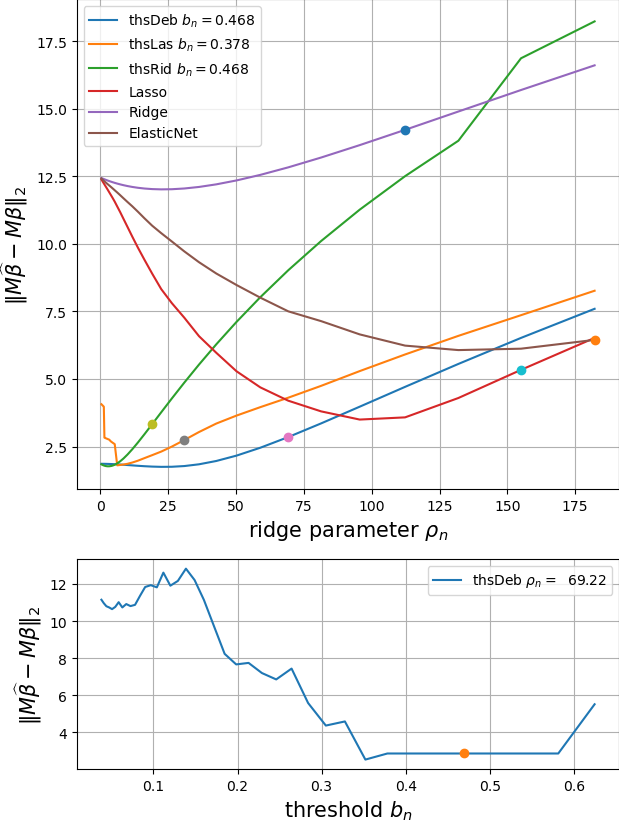}
  }
    \subfigure[Experiment 4]{
    \includegraphics[width = 2.3in]{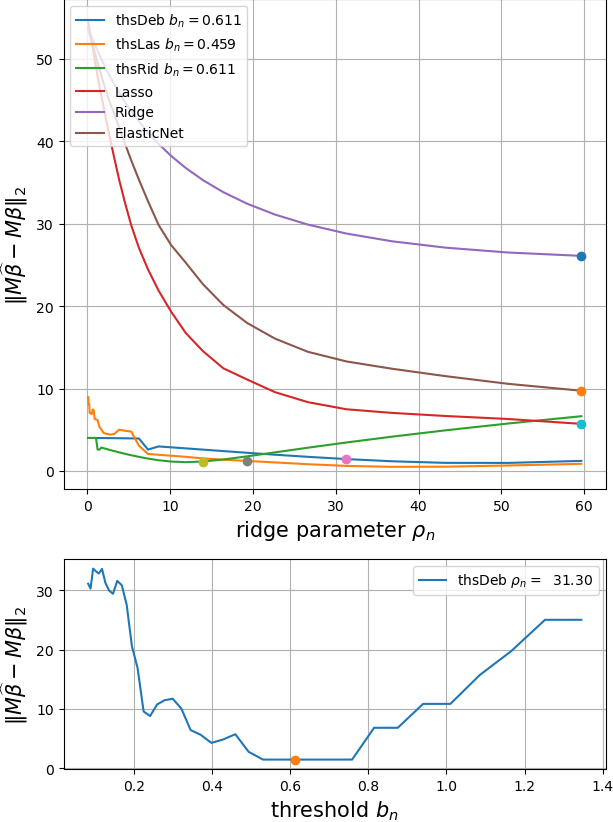}
  }
  \caption{Estimation losses of linear regression algorithms for case 1 to 4. `thsDeb' represents the debiased and threshold ridge regression method; `thsLas' represents the threshold Lasso method and `thsRid' represents the threshold ridge regression method. Dots reveal the optimal ridge parameters for each method selected by 10-fold cross validation.
  y-axis records the error $\vertiii{M\widehat{\beta} - M\beta}_2$, here $\widehat{\beta}$ is the estimator of parameters $\beta$ by different methods. We choose $l_1$ ratio $0.5$ in the ElasticNet. The small graphs below each of the graphs plot the error of the debiased and threshold ridge regression method with respect to different thresholds. $\rho_n$ and $b_n$ are chosen by 10-fold cross validation.}
  \label{Figure1}
\end{figure}

\begin{figure}[htbp]
\centering
\subfigure[Experiment 5]{
    \includegraphics[width = 2.3in]{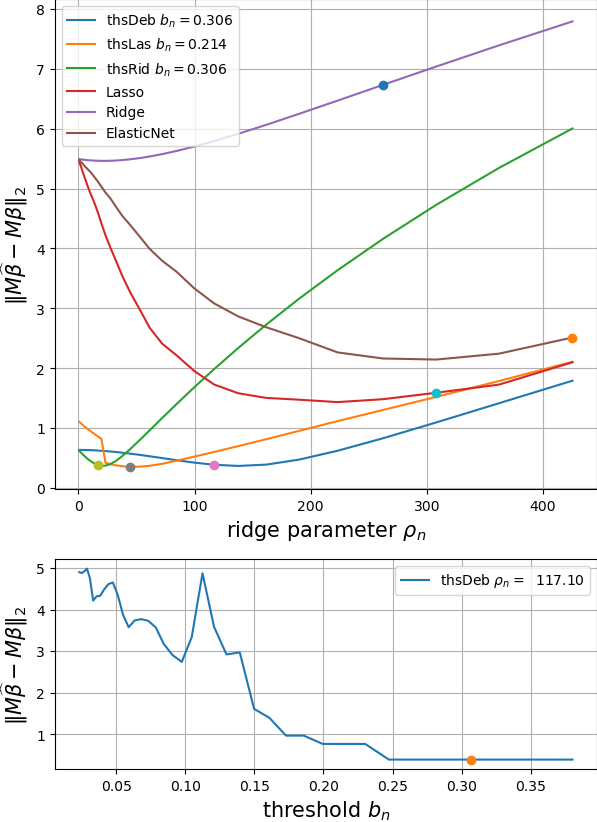}
  }
\subfigure[Experiment 6]{
    \includegraphics[width = 2.3in]{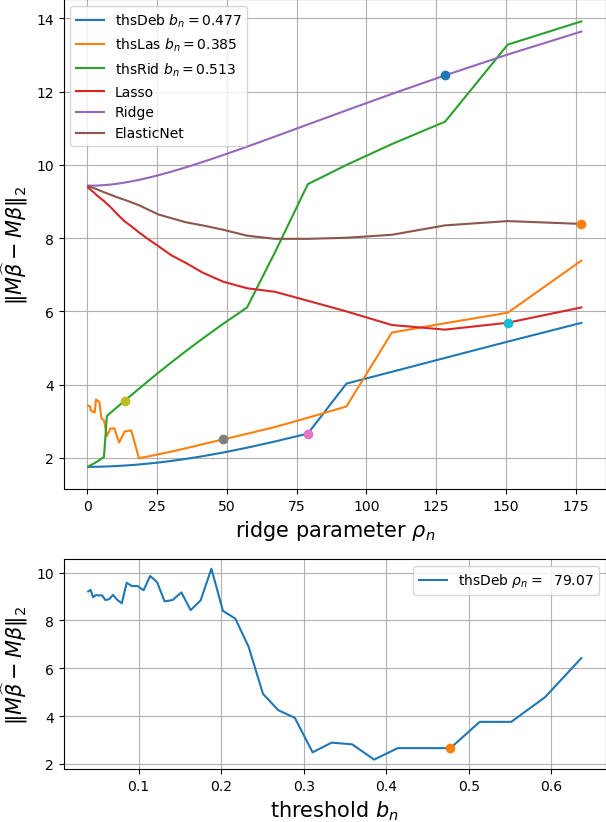}
}
\caption{Estimation losses of linear regression algorithms for case 5 to 6. Meaning of notations coincide with figure \ref{Figure1}.}
\label{figX}
\end{figure}

\begin{table}
\centering
\caption{Model selection performance of various linear regression estimators. Hyper-parameters are chosen by ten-fold cross validation. The overscore represents calculating the sample mean among 1000 simulations. `FDR' abbreviates the false discovery rate. We omit experiment 3 here for it coincides with experiment 1.}
\label{tab2}
\begin{tabular}{l l l l l l}
\hline\hline
\\
No. & Algorithm & $P\left(\widehat{\mathcal{N}}_{b_n}\neq \mathcal{N}_{b_n}\right)$ & $\overline{\widehat{\mathcal{N}}_{b_n}\Delta \mathcal{N}_{b_n}}$ & FDR & $\overline{\vertiii{X\widehat{\beta} - X\beta}}$ \\
1   &  thsDeb  & 0.211 & 0.334  & 0.010 & 11.335\\
     &  Lasso    & 0.943 & 8.663  & 0.386 & 14.992\\
     &  thsLas   & 0.190 & 0.261  & 0.005 & 10.212\\
     & Ridge     & 1.0      & 112.18 & 0.918 & 31.216\\
     & thsRid    & 0.323  & 0.411   & 0.006 & 12.369\\
     & Elastic net & 1.0  & 27.71   & 0.706  & 18.992\\
     \hline
2   &  thsDeb  & 0.790 & 1.233   & 0.014 & 21.575\\
     &  Lasso    & 1.0     & 86.28   & 0.892 & 25.873\\
     &  thsLas   & 0.142 & 0.323   & 0.015 & 10.797\\
     & Ridge     & 1.0      & 181.10 & 0.948 & 35.431\\
     & thsRid    & 0.801  & 1.214   & 0.019 & 27.727\\
     & Elastic net & 1.0  & 122.05  & 0.923 & 30.638\\
\hline
4   &  thsDeb  & 0.377 & 0.568   & 0.009 & 17.365\\
     &  Lasso    & 1.0     & 125.95 & 0.924 & 31.442\\
     &  thsLas   & 0.079 & 0.259   & 0.011 & 10.667\\
     & Ridge     & 1.0      & 284.67 & 0.966 & 46.353\\
     & thsRid    & 0.409  & 0.585   & 0.006 & 18.101\\
     & Elastic net & 1.0  & 184.87  & 0.948 & 38.673\\
\hline
5   &  thsDeb  & 0.031 & 0.045   & 0.004 & 9.781\\
     &  Lasso    & 1.0    & 32.81    & 0.728 & 18.394\\
     &  thsLas   & 0.093 & 0.248   & 0.017 & 9.703\\
     & Ridge     & 1.0     & 419.87 & 0.977 & 60.853\\
     & thsRid    & 0.035  & 0.055   & 0.004 & 9.933\\
     & Elastic net & 1.0  & 79.88  & 0.879 & 27.446\\
\hline
6   &  thsDeb  & 0.138 & 0.236   & 0.012 & 10.856\\
     &  Lasso    & 0.967    & 9.792    & 0.422 & 14.445\\
     &  thsLas   & 0.079 & 0.157   & 0.010 & 9.229\\
     & Ridge     & 1.0     &  105.775 & 0.913 & 30.164\\
     & thsRid    & 0.173  & 0.261   & 0.010 & 10.874\\
     & Elastic net & 1.0  & 28.313  & 0.711 & 18.617\\
\hline\hline
\end{tabular}
\end{table}

\subsection{Performance of bootstrap algorithm \ref{alg1}}
This section evaluates the performance of algorithm \ref{alg1} as well as the bootstrap algorithms designed for independent errors, -including Efron's bootstrap(Efron \cite{10.1214/aos/1176344552} and Chatterjee and Lahiri \cite{doi:10.1198/jasa.2011.tm10159}) and the wild bootstrap(Mammen \cite{10.1214/aos/1176349025})-, for the linear model with dependent errors. The results are demonstrated in table \ref{tabX}.

In table \ref{tabX}, the bootstrap algorithms have similar performance and are able to capture the correct $1-\alpha$ confidence intervals. Meanwhile, the dependent wild bootstrap tends to generate a wider confidence interval than Efron's bootstrap and the wild bootstrap.  Notably, the number of simultaneous linear combinations $p_1$ affects the performance of bootstrap algorithms. If $p_1$ is too large, then the bootstrap algorithms may generate an unnecessarily wide confidence interval. Numerical results show that the Efron's bootstrap and the wild bootstrap are robust to the dependent and heteroskedastic errors. However, the consistency of those algorithms cannot be assured.

\begin{table}[htbp]
\centering
\caption{Performance of various bootstrap algorithms. The number of bootstrap replicates is 1000 and the nominal coverage probability is $90\%$. We use $P$ to abbreviate the real coverage probability and $C$ to abbreviate the average $90\%$ quantile of $\max_{i = 1,...,p_1}\frac{\vert\widehat{\gamma}^*_i-\widehat{\gamma}_i\vert}{\widehat{\tau}^*_i}$. $P$ and $C$ are derived through 1000 simulations. `Efron' means `Efron's bootstrap' and `Wild' means the wild bootstrap. }
\label{tabX}
\begin{tabular}{l l l l l l l}
\hline\hline
No & Algorithm \ref{alg1} $P$ & Algorithm \ref{alg1} $C$ & Efron $P$ & Efron $C$ & Wild $P$ & Wild $C$\\
1   & $90.7\%$  & 9.32  & $88.0\%$ & 8.83  & $88.4\%$ & 8.86\\
2   & $89.1\%$  & 17.14 & $87.7\%$ & 16.71 & $87.2\%$ & 16.76\\
3   & $95.1\%$  & 11.77 & $94.4\%$ & 11.26 & $94.0\%$ & 11.31\\
4   & $89.8\%$  & 14.38 & $89.8\%$ & 13.89 & $88.8\%$ & 13.86\\
5   & $90.2\%$  & 4.72   &  $89.1\%$ & 4.62  & $89.7\%$ & 4.64\\
6   & $92.8\%$  & 11.38 & $92.2\%$ & 10.70 & $92.1\%$ & 10.75\\
\hline\hline
\end{tabular}
\end{table}

\section{Conclusion}
\label{conclusion}
This paper applies the debiased and threshold ridge regression method to a linear model $y = X\beta + \epsilon$. Then it derives a dependent wild bootstrap algorithm that constructs confidence intervals for linear combinations of parameters $\gamma = M\beta$; or tests the statistical hypothesis
\begin{equation}
M\beta = z \textit{ versus } M\beta \neq z
\label{testss}
\end{equation}
Here $M$ is a known matrix, and $z$ is the given expected value of $M\beta$.

Our work suits a high dimensional setting, i.e., the dimension of parameters $\beta$ may grow as the sample size increases. Besides, the adopted regression method and the bootstrap procedure are consistent for dependent, non-stationary, heteroskedastic errors $\epsilon$. $\epsilon$ is not necessary to be a linear process as well.

Numerical simulations indicate that the ridge regression estimator has a comparable model selection and estimation performance to complex methods like the threshold Lasso. Furthermore, it is robust to the non-optimal choice of the ridge parameter $\rho_n$ as well as the threshold $b_n$. Therefore, it can be considered as a practical method to handle real-life problems.

Our work generalizes results in Conley et al. \cite{Conley} to a more realistic setting. We have not yet seen many works on a high-dimensional linear model with non-stationary errors. So this work should bring new insights to this field.

\bibliographystyle{unsrt}
\bibliography{Draft3}

\begin{thebibliography}{10}

\bibitem{doi.org/10.1111/j.1467-9868.2005.00503.x}
Hui Zou and Trevor Hastie.
\newblock Regularization and variable selection via the elastic net.
\newblock {\em Journal of the Royal Statistical Society: Series B (Statistical
  Methodology)}, 67(2):301--320, 2005.

\bibitem{doi:10.1198/016214506000000735}
Hui Zou.
\newblock The adaptive lasso and its oracle properties.
\newblock {\em Journal of the American Statistical Association},
  101(476):1418--1429, 2006.

\bibitem{doi:10.1198/jasa.2011.tm10159}
A.~Chatterjee and S.~N. Lahiri.
\newblock Bootstrapping lasso estimators.
\newblock {\em Journal of the American Statistical Association},
  106(494):608--625, 2011.

\bibitem{10.2307/41059185}
A.~Chatterjee and S.~N. Lahiri.
\newblock Asymptotic properties of the residual bootstrap for lasso estimators.
\newblock {\em Proceedings of the American Mathematical Society},
  138(12):4497--4509, 2010.

\bibitem{10.2307/2288570}
Robert~A. Stine.
\newblock Bootstrap prediction intervals for regression.
\newblock {\em Journal of the American Statistical Association},
  80(392):1026--1031, 1985.

\bibitem{zhang2021bootstrap}
Yunyi Zhang and Dimitris~N. Politis.
\newblock Bootstrap prediction intervals with asymptotic conditional validity
  and unconditional guarantees.
\newblock (2005.09145), 2021.

\bibitem{seberLee}
George A.~F. Seber and Alan~J. Lee.
\newblock {\em Linear Regression Analysis}.
\newblock John Wiley $\&$ Sons, 2 edition, 2003.

\bibitem{VOGELSANG2012303}
Timothy~J. Vogelsang.
\newblock Heteroskedasticity, autocorrelation, and spatial correlation robust
  inference in linear panel models with fixed-effects.
\newblock {\em Journal of Econometrics}, 166(2):303--319, 2012.

\bibitem{10.1093/rfs/hhn053}
Mitchell~A. Petersen.
\newblock {Estimating Standard Errors in Finance Panel Data Sets: Comparing
  Approaches}.
\newblock {\em The Review of Financial Studies}, 22(1):435--480, 06 2008.

\bibitem{10.2307/2938229}
Donald W.~K. Andrews.
\newblock Heteroskedasticity and autocorrelation consistent covariance matrix
  estimation.
\newblock {\em Econometrica}, 59(3):817--858, 1991.

\bibitem{KIM2011349}
Min~Seong Kim and Yixiao Sun.
\newblock Spatial heteroskedasticity and autocorrelation consistent estimation
  of covariance matrix.
\newblock {\em Journal of Econometrics}, 160(2):349--371, 2011.

\bibitem{KELEJIAN2007131}
Harry~H. Kelejian and Ingmar~R. Prucha.
\newblock Hac estimation in a spatial framework.
\newblock {\em Journal of Econometrics}, 140(1):131--154, 2007.

\bibitem{doi:10.1080/07474938.2021.1874703}
Yixiao Sun and Xuexin Wang.
\newblock An asymptotically f-distributed chow test in the presence of
  heteroscedasticity and autocorrelation.
\newblock {\em Econometric Reviews}, 0(0):1--35, 2021.

\bibitem{Conley}
Timothy Conley, S\'{i}lvia Gon\c{c}alves, Min~Seong Kim, and Benoit Perron.
\newblock Bootstrap inference under cross sectional dependence.
\newblock unpublished, 2019.

\bibitem{10.2307/2346178}
Robert Tibshirani.
\newblock Regression shrinkage and selection via the lasso.
\newblock {\em Journal of the Royal Statistical Society. Series B
  (Methodological)}, 58(1):267--288, 1996.

\bibitem{10.5555/1248547.1248637}
Peng Zhao and Bin Yu.
\newblock On model selection consistency of lasso.
\newblock {\em J. Mach. Learn. Res.}, 7:2541-- 2563, 2006.

\bibitem{10.1214/009053606000000281}
Nicolai Meinshausen and Peter B\"{u}hlmann.
\newblock {High-dimensional graphs and variable selection with the Lasso}.
\newblock {\em The Annals of Statistics}, 34(3):1436 -- 1462, 2006.

\bibitem{10.1214/07-AOS582}
Nicolai Meinshausen and Bin Yu.
\newblock {Lasso-type recovery of sparse representations for high-dimensional
  data}.
\newblock {\em The Annals of Statistics}, 37(1):246 -- 270, 2009.

\bibitem{10.2307/24772752}
Cun-Hui Zhang and Stephanie~S. Zhang.
\newblock Confidence intervals for low dimensional parameters in high
  dimensional linear models.
\newblock {\em Journal of the Royal Statistical Society. Series B (Statistical
  Methodology)}, 76(1):217--242, 2014.

\bibitem{doi:10.1080/01621459.2016.1166114}
Xianyang Zhang and Guang Cheng.
\newblock Simultaneous inference for high-dimensional linear models.
\newblock {\em Journal of the American Statistical Association},
  112(518):757--768, 2017.

\bibitem{10.3150/bj/1106314846}
Eitan Greenshtein and Ya'Acov Ritov.
\newblock {Persistence in high-dimensional linear predictor selection and the
  virtue of overparametrization}.
\newblock {\em Bernoulli}, 10(6):971 -- 988, 2004.

\bibitem{statforhighdimen}
Peter B\"{u}hlmann and Sara van~de Geer.
\newblock {\em Statistics for High-Dimensional Data}.
\newblock Springer-Verlag Berlin Heidelberg, Berlin, Heidelberg, 1 edition,
  2011.

\bibitem{doi:10.1198/016214501753382273}
Jianqing Fan and Runze Li.
\newblock Variable selection via nonconcave penalized likelihood and its oracle
  properties.
\newblock {\em Journal of the American Statistical Association},
  96(456):1348--1360, 2001.

\bibitem{10.1214/15-AOS1371}
Jason~D. Lee, Dennis~L. Sun, Yuekai Sun, and Jonathan~E. Taylor.
\newblock {Exact post-selection inference, with application to the lasso}.
\newblock {\em The Annals of Statistics}, 44(3):907 -- 927, 2016.

\bibitem{liu2013}
Hanzhong Liu and Bin Yu.
\newblock Asymptotic properties of lasso+mls and lasso+ridge in sparse
  high-dimensional linear regression.
\newblock {\em Electron. J. Statist.}, 7:3124--3169, 2013.

\bibitem{10.1214/17-AOS1584}
Ryan~J. Tibshirani, Alessandro Rinaldo, Rob Tibshirani, and Larry Wasserman.
\newblock {Uniform asymptotic inference and the bootstrap after model
  selection}.
\newblock {\em The Annals of Statistics}, 46(3):1255 -- 1287, 2018.

\bibitem{shao2012}
Jun Shao and Xinwei Deng.
\newblock Estimation in high-dimensional linear models with deterministic
  design matrices.
\newblock {\em Ann. Statist.}, 40(2):812--831, 04 2012.

\bibitem{Zhang2020RidgeRR}
Yunyi Zhang and D.~N. Politis.
\newblock Ridge regression revisited: Debiasing, thresholding and bootstrap.
\newblock (arXiv:2009.08071), 2020.

\bibitem{10.1214/16-EJS1108}
Wei-Biao Wu and Ying~Nian Wu.
\newblock {Performance bounds for parameter estimates of high-dimensional
  linear models with correlated errors}.
\newblock {\em Electronic Journal of Statistics}, 10(1):352 -- 379, 2016.

\bibitem{LassoDep}
Yuefeng Han and Ruey~S. Tsay.
\newblock High-dimensional linear regression for dependent data with
  application to nowcasting.
\newblock {\em Statistica Sinica}, 30:1797--1827, 2020.

\bibitem{10.5555/17326}
Peter~J Brockwell and Richard~A Davis.
\newblock {\em Time Series: Theory and Methods}.
\newblock Springer-Verlag, Berlin, Heidelberg, 1986.

\bibitem{doi:10.1198/jasa.2009.tm08744}
Xiaofeng Shao.
\newblock The dependent wild bootstrap.
\newblock {\em Journal of the American Statistical Association},
  105(489):218--235, 2010.

\bibitem{matrix}
Roger~A. Horn and Charles~R. Johnson.
\newblock {\em Matrix Analysis}.
\newblock Cambridge University Press, 2012.

\bibitem{Wu14150}
Wei~Biao Wu.
\newblock Nonlinear system theory: Another look at dependence.
\newblock {\em Proceedings of the National Academy of Sciences},
  102(40):14150--14154, 2005.

\bibitem{wu2016}
Wei-Biao Wu and Ying~Nian Wu.
\newblock Performance bounds for parameter estimates of high-dimensional linear
  models with correlated errors.
\newblock {\em Electron. J. Statist.}, 10(1):352--379, 2016.

\bibitem{Mstat}
Jun Shao.
\newblock {\em Mathematical Statistics}.
\newblock Springer-Verlag New York, 2003.

\bibitem{10.1214/aop/1176989118}
Z.~D. Bai and Y.~Q. Yin.
\newblock {Limit of the Smallest Eigenvalue of a Large Dimensional Sample
  Covariance Matrix}.
\newblock {\em The Annals of Probability}, 21(3):1275 -- 1294, 1993.

\bibitem{RealAnalysis}
Gerald~B. Folland.
\newblock {\em Real Analysis}.
\newblock John Wiley $\&$ Sons, 2 edition, 2007.

\bibitem{10.1214/aos/1033066211}
Enno Mammen.
\newblock {Empirical process of residuals for high-dimensional linear models}.
\newblock {\em The Annals of Statistics}, 24(1):307 -- 335, 1996.

\bibitem{chernozhukov2013}
Victor Chernozhukov, Denis Chetverikov, and Kengo Kato.
\newblock Gaussian approximations and multiplier bootstrap for maxima of sums
  of high-dimensional random vectors.
\newblock {\em Ann. Statist.}, 41(6):2786--2819, 12 2013.

\bibitem{10.1214/16-AOS1512}
Danna Zhang and Wei~Biao Wu.
\newblock {Gaussian approximation for high dimensional time series}.
\newblock {\em The Annals of Statistics}, 45(5):1895 -- 1919, 2017.

\bibitem{Zhang}
Timothy Conley, S\'{i}lvia Gon\c{c}alves, Min~Seong Kim, and Benoit Perron.
\newblock Bootstrap inference under cross sectional dependence.
\newblock unpublished, 2019.

\bibitem{https://doi.org/10.1111/jtsa.12275}
Maria Fragkeskou and Efstathios Paparoditis.
\newblock Extending the range of validity of the autoregressive (sieve)
  bootstrap.
\newblock {\em Journal of Time Series Analysis}, 39(3):356--379, 2018.

\bibitem{https://doi.org/10.1111/jtsa.12106}
Paul Doukhan, Gabriel Lang, Anne Leucht, and Michael~H. Neumann.
\newblock Dependent wild bootstrap for the empirical process.
\newblock {\em Journal of Time Series Analysis}, 36(3):290--314, 2015.

\bibitem{shao}
Daisuke Kurisu, Kengo Kato, and Xiaofeng Shao.
\newblock Gaussian approximation and spatially dependent wild bootstrap for
  high-dimensional spatial data.
\newblock (arXiv:2103.10720), 2021.

\bibitem{goncalves_vogelsang_2011}
S\'{i}lvia Gon\c{c}alves and Timothy~J. Vogelsang.
\newblock Block bootstrap hac robust tests: the sophistication of the naive
  bootstrap.
\newblock {\em Econometric Theory}, 27(4):745 -- 791, 2011.

\bibitem{Suns}
Yixiao Sun and Min~Seong Kim.
\newblock Asymptotic f-test in a gmm framework with cross-sectional dependence.
\newblock {\em The Review of Economics and Statistics}, 97(1):210--223, 2015.

\bibitem{subsampling}
Dimitris~N. Politis, Joseph~P. Romano, and Michael Wolf.
\newblock {\em Subsampling}.
\newblock Springer-Verlag New York, 1999.

\bibitem{10.1214/09-SS054}
Sylvain Arlot and Alain Celisse.
\newblock {A survey of cross-validation procedures for model selection}.
\newblock {\em Statistics Surveys}, 4:40 -- 79, 2010.

\bibitem{doi:10.1081/ETC-120028836}
Dimitris~N. Politis and Halbert White.
\newblock Automatic block-length selection for the dependent bootstrap.
\newblock {\em Econometric Reviews}, 23(1):53--70, 2004.

\bibitem{Rnp}
Tristen Hayfield and Jeffrey~S. Racine.
\newblock Nonparametric econometrics: The np package.
\newblock {\em Journal of Statistical Software}, 27(5), 2008.

\bibitem{10.3150/15-BEJ756}
Arnak~S. Dalalyan, Mohamed Hebiri, and Johannes Lederer.
\newblock {On the prediction performance of the Lasso}.
\newblock {\em Bernoulli}, 23(1):552 -- 581, 2017.

\bibitem{fithian}
William Fithian, Dennis Sun, and Jonathan Taylor.
\newblock Optimal inference after model selection.
\newblock (arXiv:1410.2597), 2017.

\bibitem{10.1214/aos/1176344552}
B.~Efron.
\newblock {Bootstrap Methods: Another Look at the Jackknife}.
\newblock {\em The Annals of Statistics}, 7(1):1 -- 26, 1979.

\bibitem{10.1214/aos/1176349025}
Enno Mammen.
\newblock {Bootstrap and Wild Bootstrap for High Dimensional Linear Models}.
\newblock {\em The Annals of Statistics}, 21(1):255 -- 285, 1993.

\bibitem{StoDiff}
Bernt \o{}ksendal.
\newblock {\em Stochastic Differential Equations}.
\newblock Springer-Verlag Berlin Heidelberg, 2003.

\bibitem{burkholder1972}
D.~L. Burkholder, B.~J. Davis, and R.~F. Gundy.
\newblock Integral inequalities for convex functions of operators on
  martingales.
\newblock In {\em Proceedings of the Sixth Berkeley Symposium on Mathematical
  Statistics and Probability, Volume 2: Probability Theory}, pages 223--240,
  Berkeley, Calif., 1972. University of California Press.

\bibitem{doi:10.1137/1105028}
P.~Whittle.
\newblock Bounds for the moments of linear and quadratic forms in independent
  variables.
\newblock {\em Theory of Probability \& Its Applications}, 5(3):302--305, 1960.

\bibitem{10.1093/biomet/asz020}
Mengyu Xu, Danna Zhang, and Wei~Biao Wu.
\newblock {Pearson's chi-squared statistics: approximation theory and beyond}.
\newblock {\em Biometrika}, 106(3):716--723, 04 2019.

\bibitem{AntiConcentration}
Victor Chernozhukov, Denis Chetverikov, and Kengo Kato.
\newblock Comparison and anti-concentration bounds for maxima of gaussian
  random vectors.
\newblock {\em Probability Theory and Related Fields}, 162:47--70, 2015.

\end{thebibliography}
\appendix
\clearpage

\section{Some useful results}
\label{UsefulRes}
Notations in the appendix should coincide with section \ref{Prelim} unless we specify their meanings. This section starts with some useful corollaries of assumption 2. From corollary C.9 in \cite{StoDiff}, $\epsilon_i = \lim_{j\to\infty}\mathbf{E}\epsilon_i|\mathcal{F}_{i,j}$ almost surely; and $\lim_{j\to\infty}\Vert\epsilon_i - \mathbf{E}\epsilon_i|\mathcal{F}_{i,j}\Vert_m = 0, \forall i = 1,2,...,n$. For
\begin{equation}
\Vert\mathbf{E}\epsilon_i|\mathcal{F}_{i,j} - \mathbf{E}\epsilon_{i}|\mathcal{F}_{i,j-1}\Vert_m = \Vert\mathbf{E}(\epsilon_i - \epsilon_{i,j})|\mathcal{F}_{i,j}\Vert_m\leq \Vert\epsilon_i - \epsilon_{i,j}\Vert_m\leq \max_{i = 1,2,...,n}\delta_{i,j,m}
\label{Deltas}
\end{equation}
Therefore, assumption 2 implies $\sup_{i = 1,2,...,n; k\geq 0}(k+1)^{\alpha_\epsilon}\sum_{j = k}^\infty\Vert\mathbf{E}\epsilon_i|\mathcal{F}_{i,j} - \mathbf{E}\epsilon_{i}|\mathcal{F}_{i,j-1}\Vert_m = O(1)$. For any $a_i\in\mathbf{R}, i =1,2,...,n$ and $s\in\mathbf{Z}^+$, define $M_{i,s} = \sum_{j = n+1-i}^n a_j(\mathbf{E}\epsilon_j|\mathcal{F}_{j, s} - \mathbf{E}\epsilon_j|\mathcal{F}_{j, s-1})$, then $M_{i,s}$ is measurable in the filter $\mathcal{F}_{n, s+i-1}$. Besides, $M_{i+1,s} - M_{i,s} = a_{n-i}\mathbf{E}\epsilon_{n-i}|\mathcal{F}_{n-i,s} - a_{n-i}\mathbf{E}\epsilon_{n-i}|\mathcal{F}_{n-i, s-1}$. Apply $\pi-\lambda$ theorem to the

\noindent $\lambda-$system $\left\{A\in\mathcal{F}_{n, s+i-1}\Big|\mathbf{E}(\mathbf{E}\epsilon_{n-i}|\mathcal{F}_{n-i,s})\mathbf{1}_A = \mathbf{E}(\mathbf{E}\epsilon_{n-i}|\mathcal{F}_{n-i,s-1})\mathbf{1}_A\right\}$ and
the $\pi-$system $\{A_n\times A_{n-1}\times...\times A_{n-s-i+1}\}$. Here $A_i$ is generated by $e_i$. Then $\mathbf{E}(\mathbf{E}\epsilon_{n-i}|\mathcal{F}_{n-i,s} - \mathbf{E}\epsilon_{n-i}|\mathcal{F}_{n-i, s-1})\Big|\mathcal{F}_{n,s+i-1} = 0$ almost surely. In other words, $\{M_{i,s}\}_{i = 1,2,...,n}$ is a martingale. From Burkholder's inequality(theorem 1.1 in \cite{burkholder1972}),
\begin{equation}
\Vert M_{n,s}\Vert_m \leq C\sqrt{\Vert\sum_{j = 1}^na^2_j(\mathbf{E}\epsilon_j|\mathcal{F}_{j, s} - \mathbf{E}\epsilon_j|\mathcal{F}_{j, s-1})^2\Vert_{m/2}}\leq C\sqrt{\sum_{i = 1}^n a^2_i}\times\max_{i =1,2,...,n}\delta_{i,s,m}
\label{DeltaM}
\end{equation}
Here $C$ is independent of $s$. In particular, from theorem 2 in \cite{doi:10.1137/1105028} and assumption 2,
\begin{equation}
\begin{aligned}
\Vert\sum_{i = 1}^n a_i\epsilon_i\Vert_m\leq \Vert\sum_{i = 1}^n a_i\mathbf{E}\epsilon_i|\mathcal{F}_{i,0}\Vert_m + \sum_{s = 1}^\infty\Vert M_{n,s}\Vert_m\\
\leq C\sqrt{\sum_{i = 1}^n a^2_i\Vert\epsilon_i\Vert^2_m} + C\sqrt{\sum_{i = 1}^n a^2_i}\times \sum_{s = 1}^\infty\max_{i = 1,2,...,n}\delta_{i,s,m} = O\left(\sqrt{\sum_{i = 1}^na^2_i}\right)
\end{aligned}
\label{LinearComb}
\end{equation}
Here $C$ is a constant. Since $m>2$, \eqref{LinearComb} implies $\Sigma$'s largest eigenvalue has order $O(1)$. For $\forall 1\leq i < j \leq n$, assumption 2 and \eqref{Deltas} imply
\begin{equation}
\begin{aligned}
\vert\mathbf{E}\epsilon_i\epsilon_j\vert = \vert\mathbf{E}\epsilon_i(\epsilon_j - \mathbf{E}\epsilon_j|\mathcal{F}_{j,j-i-1})\vert\leq \Vert\epsilon_i\Vert_2\times \Vert\epsilon_j - \mathbf{E}\epsilon_j|\mathcal{F}_{j,j-i-1}\Vert_2\\
\leq \Vert\epsilon_i\Vert_m\times \sum_{s = j-i-1}^\infty \Vert\mathbf{E}\epsilon_j|\mathcal{F}_{j,s+1} - \mathbf{E}\epsilon_j|\mathcal{F}_{j,s}\Vert_m
\Rightarrow \max_{1\leq i < j\leq n}(1 + j-i)^{\alpha_\epsilon}\times \vert\mathbf{E}\epsilon_i\epsilon_j\vert = O(1)
\end{aligned}
\label{Decay}
\end{equation}
In other words, assumption 2 implies a polynomial decay on the error's covariance.

For $\forall\tau ,\psi> 0, z\in\mathbf{R}$, define $F_\tau(x_1,...,x_n) = \frac{1}{\tau}\log(\sum_{i = 1}^n \exp(\tau x_i))$; $G_\tau(x_1,...,x_n) = \frac{1}{\tau}\log(\sum_{i = 1}^n \exp(\tau x_i) + \sum_{i =1}^n \exp(-\tau x_i)) = F_\tau(x_1,...,x_n, -x_1,...,-x_n)$;

\noindent$g_0(x) =  (1 - \min(1, \max(x,0))^4)^4$; $g_{\psi, z}(x) = g_0(\psi(x - z))$; and $h_{\tau, \psi, z}(x_1,...,x_n) = g_{\psi, z}\left(G_\tau(x_1,...,x_n)\right)$. From lemma A.2 and (8) in \cite{chernozhukov2013} and (S1) to (S5) in \cite{10.1093/biomet/asz020}, $g_* = \sup_{x\in\mathbf{R}}(\vert g^{'}_0(x)\vert + \vert g^{''}_0(x)\vert + \vert g^{'''}_0(x)\vert) < \infty$; $\mathbf{1}_{x\leq z}\leq g_{\psi, z}(x)\leq \mathbf{1}_{x\leq z + 1/\psi}$; and $\sup_{x,z\in\mathbf{R}}\vert g^{'}_{\psi,z}(x)\vert\leq g_*\psi$, $\sup_{x,z\in\mathbf{R}}\vert g^{''}_{\psi, z}(x)\vert\leq g_*\psi^2$, $\sup_{x,z\in\mathbf{R}}\vert g^{'''}_{\psi,z}(x)\vert\leq g_*\psi^3$. Besides, $\frac{\partial F_\tau}{\partial x_i}\geq 0$; $\sum_{i = 1}^n\frac{\partial F_\tau}{\partial x_i} = 1$; $\sum_{i = 1}^n\sum_{j = 1}^n\vert\frac{\partial^2 F_\tau}{\partial x_i\partial x_j}\vert\leq 2\tau$; $\sum_{i = 1}^n\sum_{j = 1}^n\sum_{k = 1}^n \vert\frac{\partial^3 F_\tau}{\partial x_i\partial x_j\partial x_k}\vert\leq 6\tau^2$; and $F_\tau(x_1,...,x_n) - \frac{\log(n)}{\tau}\leq\max_{i = 1,2,...,n} x_i\leq F_\tau(x_1,...,x_n)$. Therefore,
\begin{equation}
G_\tau(x_1,...,x_n) - \frac{\log(2n)}{\tau}\leq \max_{i = 1,2,...,n}\vert x_i\vert\leq G_\tau(x_1,...,x_n)
\end{equation}
Since $\frac{\partial G_\tau}{\partial x_i} = \frac{\partial F_\tau}{\partial x_i} - \frac{\partial F_\tau}{\partial x_{i + n}}$, we get $\sum_{i = 1}^n\vert\frac{\partial G_\tau}{\partial x_i}\vert\leq 1$. For $\frac{\partial^2 G_\tau}{\partial x_i\partial x_j} = \frac{\partial^2 F_\tau}{\partial x_i\partial x_j} - \frac{\partial^2 F_\tau}{\partial x_i\partial x_{j + n}} - \frac{\partial^2 F_\tau}{\partial x_{i + n}\partial x_{j}} + \frac{\partial^2 F_\tau}{\partial x_{i+n}\partial x_{j+n}}$, we have $\sum_{i = 1}^n\sum_{j = 1}^n\vert\frac{\partial^2 G_\tau}{\partial x_i\partial x_j}\vert\leq 2\tau$. For
\begin{equation}
\begin{aligned}
\frac{\partial^3 G_\tau}{\partial x_i\partial x_j\partial x_k} = \left(\frac{\partial^3 F_\tau}{\partial x_i\partial x_j\partial x_k} + \frac{\partial^3 F_\tau}{\partial x_i\partial x_{j+n}\partial x_{k+n}} + \frac{\partial^3 F_\tau}{\partial x_{i + n}\partial x_j\partial x_{k+n}} + \frac{\partial^3 F_\tau}{\partial x_{i + n}\partial x_{j+n}\partial x_k}\right)\\
- \left(\frac{\partial^3 F_\tau}{\partial x_i\partial x_j \partial x_{n + k}} + \frac{\partial^3 F_\tau}{\partial x_i\partial x_{j + n}\partial x_k} + \frac{\partial^3 F_\tau}{\partial x_{i+n}\partial x_j\partial x_k} + \frac{\partial^3 F_\tau}{\partial x_{i+n}\partial x_{j+n}\partial x_{k+n}}\right)\\
\end{aligned}
\end{equation}
we have $\sum_{i = 1}^n\sum_{j = 1}^n\sum_{k = 1}^n\vert\frac{\partial^3 G_\tau}{\partial x_i\partial x_j\partial x_k}\vert\leq 6\tau^2$. Since $\frac{\partial h_{\tau, \psi, z}}{\partial x_i} = g^{'}_{\psi,z}(G_\tau(x_1,...,x_n))\frac{\partial G_\tau}{\partial x_i}$, we get $\sum_{i = 1}^n\vert\frac{\partial h_{\tau,\psi,z}}{\partial x_i}\vert\leq g_*\psi$.
\begin{equation}
\begin{aligned}
\frac{\partial^2 h_{\tau,\psi,z}}{\partial x_i\partial x_j} = g^{''}_{\psi,z}(G_\tau(x_1,...,x_n))\frac{\partial G_\tau}{\partial x_i}\frac{\partial G_\tau}{\partial x_j} + g^{'}_{\psi,z}(G_\tau(x_1,...,x_n))\frac{\partial^2 G_\tau}{\partial x_i\partial x_j}\\
\Rightarrow \sum_{i = 1}^n\sum_{j = 1}^n\vert \frac{\partial h_{\tau,\psi,z}}{\partial x_i\partial x_j}\vert\leq g_*\psi^2 + 2g_*\psi\tau\\
\frac{\partial^3 h_{\tau,\psi,z}}{\partial x_i\partial x_j \partial x_k} = g^{'''}_{\psi, z}(G_\tau(x_1,...,x_n))\frac{\partial G_\tau}{\partial x_i}\frac{\partial G_\tau}{\partial x_j}\frac{\partial G_\tau}{\partial x_k}+ g^{'}_{\psi, z}(G_\tau(x_1,...,x_n))\frac{\partial^3 G_\tau}{\partial x_i\partial x_j\partial x_k}\\
+ g^{''}_{\psi,z}(G_\tau(x_1,...,x_n))\left(\frac{\partial^2 G_\tau}{\partial x_i\partial x_k}\frac{\partial G_\tau}{\partial x_j} + \frac{\partial^2 G_\tau}{\partial x_j\partial x_k}\frac{\partial G_\tau}{\partial x_i} + \frac{\partial^2 G_\tau}{\partial x_i\partial x_j}\frac{\partial G_\tau}{\partial x_k}\right) \\
\Rightarrow \sum_{i = 1}^n\sum_{j = 1}^n\sum_{k = 1}^n\vert\frac{\partial^3 h_{\tau,\psi,z}}{\partial x_i\partial x_j \partial x_k}\vert\leq g_*\psi^3 + 6g_*\psi\tau^2 + 6g_*\tau\psi^2
\end{aligned}
\label{allInequ}
\end{equation}
\section{Proof of theorem \ref{thm1} and theorem \ref{Central}}
\label{app2}
We first prove necessary lemmas.
\begin{lemma}
Suppose random variables $\epsilon_i, i = 1,2,...,n$ satisfy assumption 2, and $\{\gamma_{ij}\}_{i = 1,2,...,k, j = 1,2,...,n}\in\mathbf{R}^{k\times n}$ satisfy

\noindent $\max_{i = 1,2,...,k}\sum_{j = 1}^n\gamma^2_{ij}\leq D^2$. Then
\begin{equation}
\max_{i = 1,2,...,k}\vert\sum_{j = 1}^n\gamma_{ij}\epsilon_j\vert = O_p(k^{1/m}D)
\label{ProbIneq}
\end{equation}
\label{lemmaMoment}
\end{lemma}
\begin{proof}
Form \eqref{LinearComb}, $\forall \xi>0$,
\begin{equation}
\begin{aligned}
Prob\left(\max_{i = 1,2,...,k}\vert\sum_{j = 1}^n\gamma_{ij}\epsilon_j\vert>\xi\right)\leq \frac{1}{\xi^m}\sum_{i = 1}^k\Vert\sum_{j = 1}^n\gamma_{ij}\epsilon_j\Vert_m^m = O\left(\frac{k}{\xi^m}\times \left(\max_{i =1,2,...,k}\sum_{j = 1}^n\gamma^2_{ij}\right)^{m/2}\right)
\end{aligned}
\end{equation}
\eqref{ProbIneq} is proved by choosing $\xi = Ck^{1/m}D$ with a sufficiently large constant $C$.
\end{proof}
Lemma \ref{lemmaMoment} lays the theoretical foundation for the threshold ridge regression estimator $\widehat{\beta}$'s (model selection) consistency. Lemma \ref{GaussianRes} is a corollary of Chernozhukov et al. \cite{AntiConcentration}. It introduces some properties of a Gaussian random vector.

\begin{lemma}
1. Suppose $\epsilon = (\epsilon_1,...,\epsilon_n)^T$ are joint normal random variables, $\mathbf{E}\epsilon = 0$, $\mathbf{E}\epsilon\epsilon^T = \Sigma = (\sigma_{ij})_{i,j = 1,...,n}$ is non-singular; and $\exists 0<c_0\leq C_0<\infty$ such that $c_0\leq\sigma_{ii}\leq C_0, i = 1,2,...,n$. Then
\begin{equation}
\sup_{x\in\mathbf{R}}\vert Prob(\max_{i = 1,2,...,n}\vert\epsilon_i\vert\leq x+\delta) - Prob(\max_{i = 1,2,...,n}\vert\epsilon_i\vert\leq x)\vert \leq C\delta(1 + \sqrt{\log(n)} + \sqrt{\vert\log(\delta)\vert})
\label{FirGaussian}
\end{equation}
Here $C$ only depends on $c_0$ and $C_0$.

\noindent 2. In addition suppose $\xi = (\xi_1,...,\xi_n)^T$ are joint normal random variables with $\mathbf{E}\xi = 0$, $\mathbf{E}\xi\xi^T = \Sigma^* = (\sigma^*_{ij})_{i,j = 1,2,...,n}$, and $\Delta = \max_{i,j = 1,2,...,n}\vert\sigma_{ij} - \sigma^*_{ij}\vert < 1$. Then
\begin{equation}
\begin{aligned}
\sup_{x\in\mathbf{R}}\vert Prob(\max_{i = 1,2,...,n}\vert\epsilon_i\vert\leq x) - Prob(\max_{i = 1,2,...,n}\vert\xi_i\vert\leq x)\vert\\
\leq C^*\left(\Delta^{1/3}(1 + \log^3(n)) + \frac{\Delta^{1/6}}{1 + \log^{1/4}(n)}\right)
\label{Continuity}
\end{aligned}
\end{equation}
for $n = 1,2,...$. Here $C^*$ only depends on $c_0$ and  $C_0$.
\label{GaussianRes}
\end{lemma}
We emphasize that $\Sigma^*$ can be singular. From theorem 4.1.5 in \cite{matrix} $\Sigma^* = Q\Lambda Q^T, \Lambda = diag(\lambda_1,...,\lambda_r, 0, 0, ..., 0), 0\leq r<n$, and $Q = (q_{ij})_{i,j = 1,...,n}$ satisfies $QQ^T = Q^TQ = I_n$. We define $\tau = Q^T\xi = (\tau_1,...,\tau_r, 0, ..., 0)^T$ almost surely. For any continuously differentiable function $f$ such that $\mathbf{E}\sum_{i = 1}^n \vert\frac{\partial f}{\partial x_i}\vert<\infty$, lemma 2 in \cite{chernozhukov2013} implies
\begin{equation}
\mathbf{E}\xi_i f(\xi) = \sum_{j = 1}^r q_{ij}\mathbf{E}\tau_j f(Q\tau) = \sum_{j = 1}^r \sum_{k = 1}^r \sum_{l = 1}^n q_{ij}q_{lk}\mathbf{E}(\tau_j\tau_k)\times\mathbf{E}\frac{\partial f}{\partial x_l}(Q\tau) = \sum_{l = 1}^n\mathbf{E}\frac{\partial f}{\partial x_l}(\xi)\mathbf{E}\xi_i\xi_l
\label{Stein}
\end{equation}
This observation assures that \eqref{Continuity} works for degenerated $\xi$.
\begin{proof}
Since $\vert\epsilon_i\vert = \max(\epsilon_i,-\epsilon_i)\Rightarrow \max_{i = 1,2,...,n}\vert\epsilon_i\vert = \max(\max_{i = 1,...,n}\epsilon_i, \max_{i = 1,2,...,n}-\epsilon_i)$, and $-\epsilon$ has the same distribution as $\epsilon$,
\begin{equation}
\begin{aligned}
\sup_{x\in\mathbf{R}}\left(Prob\left(\max_{i = 1,2,...,n}\vert\epsilon_i\vert\leq x+\delta\right) - Prob\left(\max_{i = 1,2,...,n}\vert\epsilon_i\vert\leq x\right)\right)
\leq \sup_{x\in\mathbf{R}}Prob\left(x<\max_{i = 1,...,n}\epsilon_i\leq x+\delta\right)\\
+ \sup_{x\in\mathbf{R}}Prob\left(x<\max_{i = 1,2,...,n}-\epsilon_i\leq x+\delta\right)
\leq 2\sup_{x\in\mathbf{R}}Prob\left(\vert \max_{i = 1,...,n}\epsilon_i - x\vert\leq \delta\right)
\end{aligned}
\end{equation}
From theorem 3 and (18), (19) in Chernozhukov et al.\cite{AntiConcentration}(also see lemma 3 in Zhang and Politis \cite{Zhang2020RidgeRR}), we define $\underline{\sigma} = \min_{i = 1,2,...,n}\sigma_{ii}$ and $\overline{\sigma} = \max_{i = 1,2,...,n}\sigma_{ii}$,
\begin{equation}
\begin{aligned}
\sup_{x\in\mathbf{R}}Prob\left(\vert\max_{i=1,2,...,n}\epsilon_i - x\vert\leq \delta\right)\leq \frac{\sqrt{2}\delta}{\underline{\sigma}}\left(\sqrt{\log(n)} + \sqrt{\max(1, \log(\underline{\sigma}) - \log(\delta))}\right)\\
+ \frac{4\sqrt{2}\delta}{\underline{\sigma}}\times\left(\frac{\overline{\sigma}}{\underline{\sigma}}\sqrt{\log(n)} + 2 + \frac{\overline{\sigma}}{\underline{\sigma}}\sqrt{\max(0,\log(\underline{\sigma}) - \log(\delta))}\right)\\
\leq \frac{\sqrt{2}\delta}{c_0}\left(\sqrt{\log(n)} + \sqrt{1 + \vert\log(c_0)\vert + \vert\log(C_0)\vert} + \sqrt{\vert\log(\delta)\vert}\right)\\
+ \frac{4\sqrt{2}\delta C_0}{c_0^2}\left(\sqrt{\log(n)} + 2 + \sqrt{\vert\log(c_0)\vert + \vert\log(C_0)\vert} + \sqrt{\vert\log(\delta)\vert}\right)\\
\leq \left(\frac{\sqrt{2\times (1 + \vert\log(c_0)\vert + \vert \log(C_0)\vert)}}{c_0} + \frac{4\sqrt{2}C_0}{c_0^2}(2 + \sqrt{\vert\log(c_0)\vert+\vert\log(C_0)\vert})\right)\\
\times \delta\left(\sqrt{\log(n)} + 1 + \sqrt{\vert\log(\delta)\vert}\right)
\end{aligned}
\label{C_0}
\end{equation}
and we prove \eqref{FirGaussian}.

Without loss of generality, assume $\epsilon$ is independent of $\xi$. Similar to Chernozhukov et al.\cite{AntiConcentration}, for any $0\leq t\leq 1$, we define random variables $Z_i(t) = \sqrt{t}\epsilon_i + \sqrt{1 - t}\xi_i$. According to \eqref{allInequ}, \eqref{Stein}, and theorem 2.27 in Folland \cite{RealAnalysis},
\begin{equation}
\begin{aligned}
\mathbf{E}h_{\tau, \psi, x}(\epsilon_1,...,\epsilon_n) - \mathbf{E}h_{\tau, \psi, x}(\xi_1,...,\xi_n)\\
=\frac{1}{2}\sum_{i = 1}^n\int_{0}^1 t^{-1/2}\mathbf{E}(\frac{\partial h_{\tau, \psi,x}(Z_1(t),...Z_n(t))}{\partial x_i}\epsilon_i)dt - \frac{1}{2}\sum_{i = 1}^n\int_0^1(1 - t)^{-1/2}\mathbf{E}(\frac{\partial h_{\tau,\psi,x}(Z_1(t),...,Z_n(t))}{\partial x_i}\xi_i)dt\\
=\frac{1}{2}\sum_{i = 1}^n\sum_{k = 1}^n(\sigma_{ik} - \sigma^*_{ik})\int_0^1\mathbf{E}\frac{\partial^2 h_{\tau,\psi,x}(Z_1(t),...,Z_n(t))}{\partial x_i\partial x_k}dt\\
\Rightarrow \vert\mathbf{E}h_{\tau, \psi, x}(\epsilon_1,...,\epsilon_n) - \mathbf{E}h_{\tau, \psi, x}(\xi_1,...,\xi_n)\vert\leq \Delta\times(g_*\psi^2 + g_*\psi\tau)
\end{aligned}
\label{Hs}
\end{equation}
For any $x\in\mathbf{R}$ and given $\tau,\psi > 0$, define $t = \frac{1}{\psi} + \frac{\log(2n)}{\tau}$, then
\begin{equation}
\begin{aligned}
Prob\left(\max_{i = 1,2,...,n}\vert\epsilon_i\vert\leq x\right) - Prob\left(\max_{i = 1,2,...,n}\vert\xi_i\vert\leq x\right)\\
\leq Prob\left(\max_{i = 1,2,...,n}\vert\epsilon_i\vert\leq x - t\right) + Ct(\sqrt{\log(n)} + \sqrt{\vert\log(t)\vert} + 1) - Prob\left(\max_{i = 1,2,...,n}\vert\xi_i\vert\leq x\right)\\
\leq \mathbf{E}h_{\tau, \psi, x - \frac{1}{\psi}}(\epsilon_1,...,\epsilon_n) - \mathbf{E}h_{\tau, \psi, x - \frac{1}{\psi}}(\xi_1,...,\xi_n)+ Ct(\sqrt{\log(n)} + \sqrt{\vert\log(t)\vert} + 1)\\
Prob\left(\max_{i = 1,2,...,n}\vert\epsilon_i\vert\leq x\right) - Prob\left(\max_{i = 1,2,...,n}\vert\xi_i\vert\leq x\right)\\
\geq Prob\left(\max_{i = 1,2,...,n}\vert\epsilon_i\vert\leq x + t\right) - Ct( \sqrt{\log(n)} + \sqrt{\vert\log(t)\vert} + 1) - Prob\left(\max_{i = 1,2,...,n}\vert\xi_i\vert\leq x\right)\\
\geq \mathbf{E}h_{\tau, \psi, x + \frac{\log(2n)}{\tau}}(\epsilon_1,...,\epsilon_n) - \mathbf{E}h_{\tau, \psi, x + \frac{\log(2n)}{\tau}}(\xi_1,...,\xi_n)- Ct(\sqrt{\log(n)} + \sqrt{\vert\log(t)\vert} + 1)\\
\Rightarrow \sup_{x\in\mathbf{R}}\vert Prob\left(\max_{i = 1,2,...,n}\vert\epsilon_i\vert\leq x\right) - Prob\left(\max_{i = 1,2,...,n}\vert\xi_i\vert\leq x\right)\vert\\
\leq \Delta\times(g_*\psi^2 + g_*\psi\tau) +  Ct(\sqrt{\log(n)} + \sqrt{\vert\log(t)\vert} + 1)
\end{aligned}
\label{Separa}
\end{equation}
We choose $\tau = \psi = \left(1 + \log^{3/2}(n)\right)/\Delta^{1/3}$, then $\exists$ a constant $C _1> 0$ such that $\frac{1}{C_1}\frac{\Delta^{1/3}}{1 + \log^{1/2}(n)}\leq t = \Delta^{1/3}\left(\frac{1 + \log(2)}{1 + \log^{3/2}(n)}\right) + \frac{\Delta^{1/3}\times \log(n)}{1 + \log^{3/2}(n)}\leq \frac{C_1\Delta^{1/3}}{1 + \log^{1/2}(n)}$ for $n = 1,2,...$;  and we prove \eqref{Continuity}.
\end{proof}
Lemma \ref{GaussianApprox} approximates the distribution of the linear combinations of errors $\epsilon$ by a multivariate normal random vector. This lemma can be used to derive the Gaussian approximation theorem for the estimator $\widehat{\beta}$.

\begin{lemma}[Gaussian approximation theorem]
\label{GaussianApprox}
Suppose $\epsilon_i, i = 1,2,...,n$ satisfy assumption 2, and $(\gamma_{ij})_{i =1,...,p_1, j =1,...,n}\in\mathbf{R}^{p_1\times n}$ has rank $p_1$, $p_1 = O(1),\ p_1\leq n$. Besides, suppose $\exists$ constants $ 0<c_\gamma\leq C_\gamma<\infty$ such that $c_\gamma\leq\sum_{j = 1}^n\gamma^2_{ij}\leq C_\gamma$ for $\forall i =1,2,...,p_1$  and $n = 1,2,...$. Define $\gamma^* = \max_{i = 1,...,p_1, j =1,...,n}\vert\gamma_{ij}\vert $ and  assume $\gamma^* = o(n^{-1/4}\log^{-z}(n))$, here $z = \max(\frac{9}{2}, \frac{3\alpha_\epsilon}{2\alpha_\epsilon-2})$. Then
\begin{equation}
\sup_{x\in\mathbf{R}}\vert Prob\left(\max_{i = 1,2,...,p_1}\vert\sum_{j = 1}^n\gamma_{ij}\epsilon_j\vert\leq x\right) - Prob\left(\max_{i = 1,2,...,p_1}\vert\sum_{j = 1}^n \gamma_{ij}\xi_j\vert\leq x\right)\vert = o(1)
\label{GauX}
\end{equation}
Here $\xi = (\xi_1,...,\xi_n)^T$ has multivariate normal distribution, $\mathbf{E}\xi = 0$ and $\mathbf{E}\xi\xi^T  = \Sigma$, the covariance matrix of $\epsilon$.
\end{lemma}
\begin{proof}
For any given $\psi, \tau > 0$, define $t = \frac{1}{\psi} + \frac{\log(2p_1)}{\tau}$. Assumption 2, \eqref{LinearComb}, and \eqref{Separa} imply
\begin{equation}
\begin{aligned}
\sup_{x\in\mathbf{R}}\vert Prob\left(\max_{i = 1,2,...,p_1}\vert\sum_{j = 1}^n\gamma_{ij}\epsilon_j\vert\leq x\right) - Prob\left(\max_{i = 1,2,...,p_1}\vert\sum_{j = 1}^n\gamma_{ij}\xi_j\vert\leq x\right)\vert\\
\leq \sup_{x\in\mathbf{R}}\vert\mathbf{E}h_{\tau,\psi,x}(\sum_{j = 1}^n\gamma_{1j}\epsilon_j,...,\sum_{j = 1}^n\gamma_{p_1j}\epsilon_j) - \mathbf{E}h_{\tau,\psi,x}(\sum_{j = 1}^n\gamma_{1j}\xi_j,...,\sum_{j = 1}^n\gamma_{p_1j}\xi_j)\vert\\
+ O\left(t(\sqrt{\log(p_1)} + \sqrt{\vert\log(t)\vert} + 1)\right)
\end{aligned}
\label{ProofHead}
\end{equation}
For any integer $s > 0$, \eqref{DeltaM} implies
\begin{equation}
\begin{aligned}
\Vert\sum_{j = 1}^n\gamma_{ij}(\epsilon_j - \mathbf{E}\epsilon_j|\mathcal{F}_{j,s})\Vert_m\leq \sum_{k = s}^\infty\Vert\sum_{j = 1}^n\gamma_{ij}(\mathbf{E}\epsilon_j|\mathcal{F}_{j,k + 1} - \mathbf{E}\epsilon_j|\mathcal{F}_{j, k})\Vert_m\\
\leq C\sqrt{\sum_{j = 1}^n\gamma^2_{ij}}\sum_{k = s}^\infty\max_{j = 1,2,...,n}\delta_{j,k,m}\\
\Rightarrow \max_{i = 1,2,...,p_1}\Vert\sum_{j = 1}^n\gamma_{ij}(\epsilon_j - \mathbf{E}\epsilon_j|\mathcal{F}_{j,s})\Vert_m = O\left(\frac{1}{(s + 1)^{\alpha_\epsilon}}\right)
\end{aligned}
\end{equation}
Here $C$ is a constant. Therefore,
\begin{equation}
\begin{aligned}
\sup_{x\in\mathbf{R}}\vert\mathbf{E}h_{\tau,\psi,x}(\sum_{j = 1}^n\gamma_{1j}\epsilon_j,...,\sum_{j = 1}^n\gamma_{p_1j}\epsilon_j) - \mathbf{E}h_{\tau,\psi,x}(\sum_{j = 1}^n\gamma_{1j}\mathbf{E}\epsilon_j|\mathcal{F}_{j,s},...,\sum_{j = 1}^n\gamma_{p_1j}\mathbf{E}\epsilon_j|\mathcal{F}_{j,s})\vert\\
\leq g_*\psi\mathbf{E}\max_{i = 1,2,...,p_1}\vert\sum_{j = 1}^n\gamma_{ij}(\epsilon_j - \mathbf{E}\epsilon_j|\mathcal{F}_{j,s})\vert\\
\leq g_*\psi\sum_{i = 1}^{p_1}\Vert\sum_{j = 1}^n\gamma_{ij}(\epsilon_j - \mathbf{E}\epsilon_j|\mathcal{F}_{j,s})\Vert_m = O\left(\frac{\psi}{(s+1)^{\alpha_\epsilon}}\right)
\end{aligned}
\label{Firpart}
\end{equation}
For any integer $k > s$, define the big block $S_{il} = \sum_{j = (l - 1)\times(k+s) + 1}^{(l - 1)\times(k+s) +k}\gamma_{ij}\mathbf{E}\epsilon_j|\mathcal{F}_{j,s}$ and the small block $s_{il} = \sum_{j = (l - 1)\times (k+s) + k + 1}^{l\times(k+s)}\gamma_{ij}\mathbf{E}\epsilon_j|\mathcal{F}_{j,s}$ for $l = 1,2,..., r = \lfloor \frac{n}{k+s}\rfloor$. Here $\lfloor x\rfloor$ denotes the largest integer that is smaller than or equal to $x$. If $(k+s)r + k \geq n$, we define $S_{i(r+1)} = \sum_{j = (k+s)r + 1}^n\gamma_{ij}\mathbf{E}\epsilon_j|\mathcal{F}_{j,s}$ and $s_{i(r + 1)} = 0$. If $(k+s)r + k < n$, we define $S_{i(r+1)} = \sum_{j = (k+s)r + 1}^{j = (k+s)r + k}\gamma_{ij}\mathbf{E}\epsilon_j|\mathcal{F}_{j,s}$ and $s_{i(r+1)} = \sum_{j = (k+s)r + k + 1}^{n}\gamma_{ij}\mathbf{E}\epsilon_j|\mathcal{F}_{j,s}$. $(S_{1l},...,S_{p_1l})^T,\ l = 1,2,...,r + 1$ are independent with each other; $(s_{1l},...,s_{p_1l})^T, l = 1,2,...,r+1$ are independent with each other; and $\sum_{j = 1}^n\gamma_{ij}\mathbf{E}\epsilon_j|\mathcal{F}_{j,s} = \sum_{j = 1}^{r+1}(S_{ij} +s_{ij})$. For $\mathbf{E}\epsilon_j|\mathcal{F}_{j,s} = \mathbf{E}\epsilon_j|\mathcal{F}_{j,0} + \sum_{l = 1}^s \mathbf{E}\epsilon_j|\mathcal{F}_{j,l} - \mathbf{E}\epsilon_j|\mathcal{F}_{j,l-1}$, \eqref{DeltaM} and \eqref{LinearComb} imply
\begin{equation}
\begin{aligned}
\vert\mathbf{E}h_{\tau, \psi, x}(\sum_{j = 1}^{r+1}S_{1j}+s_{1j},...,\sum_{j=1}^{r+1}S_{p_1j} + s_{p1j}) - \mathbf{E}h_{\tau,\psi,x}(\sum_{j = 1}^{r + 1}S_{1j},...,\sum_{j = 1}^{r+1}S_{p_1j})\vert\\
\leq g_*\psi\sum_{i = 1}^{p_1}\Vert\sum_{j = 1}^{r+1}s_{ij}\Vert_m
=O\left(\psi\max_{i = 1,...,p_1}\sqrt{\sum_{j = 1}^{r+1}\sum_{l = (j-1)\times(k+s)+k+1}^{j\times(k+s)}\gamma^2_{il}}\right)\\
= O\left(\psi\times\gamma^*\sqrt{\frac{ns}{k}}\right)
\end{aligned}
\label{Secc}
\end{equation}
Here we define $\gamma_{ij} = 0$ for $i = 1,2,...,p_1, j>n$.

Define $\xi_{j,s}, j = 1,2,...,n$ as the joint normal random variables with mean $0$ and the same covariance matrix as $\mathbf{E}\epsilon_j|\mathcal{F}_{j,s},\ j = 1,2,...,n$; and are independent of $e_i, i = ...,-1,0,1,...$. Define $S^*_{il} = \sum_{j = (l - 1)(k+s)+1}^{(l-1)(k+s) + k}\gamma_{ij}\xi_{j,s}$ for $l = 1, ..., r$ and $S^*_{ir+1} = \sum_{j = r(k+s)+1}^{\min(n, r(k+s)+k)}\gamma_{ij}\xi_{j,s}$. Define $H_{il} = \sum_{j = 1}^{l-1}S_{ij} + \sum_{j = l+1}^{r+1}S^*_{ij}$, then $H_{il} + S_{il} = H_{il+1}+S_{il+1}^*$.  $S^*_{il}, i = 1,...,p_1, l = 1,...,r+1$ has joint normal distribution(we treat a constant as a degenerated normal random variable) and $\mathbf{E}S^*_{il}S^*_{jk} = \mathbf{E}S_{il}S_{jk} = 0$ for $l\neq k$, so $(S^*_{1l},...,S^*_{p_1l})^T$ are independent with each other. From Taylor's theorem,
\begin{equation}
\begin{aligned}
\vert\mathbf{E}\left(h_{\tau,\psi,x}(H_{1j} + S_{1j},...,H_{p_1j} + S_{p_1j}) - h_{\tau,\psi,x}(H_{1j} + S^*_{1j},...,H_{p_1j} + S^*_{p_1j})\right)|H_{1j},...,H_{p_1j}\vert\\
\leq\vert\sum_{l = 1}^{p_1}\frac{\partial h_{\tau,\psi,x}(H_{1j},...,H_{p_1j})}{\partial x_l}\mathbf{E}(S_{lj} - S^*_{lj})\vert\\ + \vert\frac{1}{2}\sum_{l_1 = 1}^{p_1}\sum_{l_2 = 1}^{p_1}\frac{\partial^2 h_{\tau,\psi,x}(H_{1j},...,H_{p_1j})}{\partial x_{l_1}\partial x_{l_2}}\mathbf{E}(S_{l_1j}S_{l_2j} - S^*_{l_1j}S^*_{l_2j})\vert\\
+g_*(\psi^3 + \tau\psi^2 + \psi\tau^2)\mathbf{E}(\max_{i = 1,2,...,p_1}\vert S_{ij}\vert^3 + \max_{i = 1,2,...,p_1}\vert S_{ij}^*\vert^3)\\
\leq g_*(\psi^3 + \tau\psi^2 + \psi\tau^2)\sum_{i = 1}^{p_1}\left(\Vert S_{ij}\Vert_m^3 + \Vert S^*_{ij}\Vert_m^3\right)\\
\leq C\left((\psi^3 + \tau\psi^2 + \psi\tau^2)\times \max_{i =1,2,...,p_1}\left(\sum_{l = (j - 1)(k+s)+1}^{(j-1)(k+s)+k}\gamma^2_{il}\right)^{3/2}\right)
\end{aligned}
\end{equation}
Here $C$ is a constant. Since $\max_{i = 1,...,p_1}\sum_{j = 1}^n\gamma^2_{ij}\leq C_\gamma$,
\begin{equation}
\begin{aligned}
\sup_{x\in\mathbf{R}}\vert\mathbf{E}h_{\tau,\psi,x}(\sum_{j = 1}^{r+1}S_{1j},...,\sum_{j = 1}^{r+1}S_{p_1j}) - \mathbf{E}h_{\tau, \psi, x}(\sum_{j = 1}^{r+1}S_{1j}^*,...,\sum_{j = 1}^{r+1}S_{p_1j}^*)\vert\\
\leq \sum_{j = 1}^{r+1}\sup_{x\in\mathbf{R}}\vert \mathbf{E}h_{\tau,\psi,x}(H_{1j} + S_{1j},...,H_{p_1j} + S_{p_1j}) - \mathbf{E}h_{\tau,\psi,x}(H_{1j} + S^*_{1j},...,H_{p_1j} + S^*_{p_1j})\vert\\
\leq C(\psi^3 + \tau\psi^2 + \psi\tau^2)\sum_{j = 1}^{r+1}\sum_{i = 1}^{p_1}\left(\sum_{l = (j - 1)(k+s)+1}^{(j-1)(k+s)+k}\gamma^2_{il}\right)^{3/2} \\
= O\left((\psi^3 + \tau\psi^2 + \psi\tau^2)\times \max_{i =1,...,p_1,j=1,...,r+1}\sqrt{\sum_{l = (j - 1)(k+s)+1}^{(j-1)(k+s)+k}\gamma^2_{il}}\right)
\end{aligned}
\end{equation}
which has order $O\left((\psi^3 + \tau\psi^2 + \psi\tau^2)\times \gamma^*\sqrt{k}\right)$. We define $s^*_{il} = \sum_{j = (l-1)(k+s)+k+1}^{l(k+s)}\gamma_{ij}\xi_{j,s}$ for $l = 1,2,...,r$, and $s^*_{ir+1} = \sum_{j = r(k+s)+k+1}^n\gamma_{ij}\xi_{j,s}$.
For $s^*_{il}$ has normal distribution,
\begin{equation}
\begin{aligned}
\vert\mathbf{E}h_{\tau,\psi,x}(\sum_{j = 1}^{r+1}S^*_{1j},...,\sum_{j = 1}^{r+1}S^*_{p_1j}) - \mathbf{E}h_{\tau,\psi,x}(\sum_{j = 1}^n\gamma_{1j}\xi_{j,s},...,\sum_{j=1}^n\gamma_{p_1j}\xi_{j,s})\vert\\
\leq g_*\psi\sum_{i = 1}^{p_1}\Vert\sum_{j = 1}^{r+1}s_{ij}^*\Vert_m
=O\left(\psi\max_{i = 1,2,...,p_1}\Vert\sum_{j = 1}^{r+1}s_{ij}\Vert_2\right)
\end{aligned}
\end{equation}
According to \eqref{Secc}, this has order $O\left(\psi\times\gamma^*\sqrt{\frac{ns}{k}}\right)$. From \eqref{Hs}, \eqref{Decay}, and 0.9.7 in \cite{matrix},
\begin{equation}
\begin{aligned}
\vert\mathbf{E}h_{\beta,\psi,x}(\sum_{j = 1}^n\gamma_{1j}\xi_{j,s},...,\sum_{j=1}^n\gamma_{p_1j}\xi_{j,s}) - \mathbf{E}h_{\beta, \psi, x}(\sum_{j = 1}^n\gamma_{1j}\xi_j,...,\sum_{j=1}^n\gamma_{p_1j}\xi_j)\vert\\
\leq (g_*\psi^2 + g_*\psi\tau)\times\max_{i,j=1,2,...,p_1}\vert\sum_{k_1=1}^n\sum_{k_2=1}^n\gamma_{ik_1}\gamma_{jk_2}(\mathbf{E}(\mathbf{E}\epsilon_{k_1}|\mathcal{F}_{k_1,s}\times \mathbf{E}\epsilon_{k_2}|\mathcal{F}_{k_2,s}) - \sigma_{k_1k_2})\vert\\
\leq (g_*\psi^2 + g_*\psi\tau)\max_{i,j=1,...,p_1}\sum_{\vert k_1 - k_2\vert\leq s}\vert\gamma_{ik_1}\gamma_{jk_2}\vert(\Vert\mathbf{E}\epsilon_{k_1}|\mathcal{F}_{k_1,s}\Vert_2\Vert\mathbf{E}\epsilon_{k_2}|\mathcal{F}_{k_2,s} - \epsilon_{k_2}\Vert_2 + \Vert\mathbf{E}\epsilon_{k_1}|\mathcal{F}_{k_1,s} - \epsilon_{k_1}\Vert_2\Vert\epsilon_{k_2}\Vert_2)\\
+C(g_*\psi^2 + g_*\psi\tau)\max_{i,j=1,...,p_1}\sum_{\vert k_1 - k_2\vert>s}\frac{\vert\gamma_{ik_1}\gamma_{jk_2}\vert}{(1 + \vert k_1 - k_2\vert)^{\alpha_\epsilon}}\\
= O\left((\psi^2+\psi\tau)\times\left(\max_{i =1,2,...,p_1}\sum_{k = 1}^n\gamma_{ik}^2\right)\times\left(\frac{1}{(1 + s)^{\alpha_\epsilon-1}}\right)\right)
\end{aligned}
\label{Zetas}
\end{equation}
Here $\sigma_{ij} = \mathbf{E}\epsilon_i\epsilon_j$.

Define $V$ such that $\frac{1}{V} = \gamma^*\times n^{1/4}\log^z(n)\to 0$. Choose $k = \lfloor\sqrt{n}\rfloor$, $\psi= \tau = V^{(\alpha_\epsilon-1)/(3\alpha_\epsilon+3)}$, and $s = \lfloor V^{2/(\alpha_\epsilon+1)}\log^{3/(\alpha_\epsilon-1)}(n)\rfloor$. Then we prove \eqref{GauX}.
\end{proof}

\begin{proof}[proof of theorem \ref{thm1}]
\begin{equation}
\begin{aligned}
\max_{i = 1,2,...,p}\sum_{l = 1}^n \left(\sum_{j = 1}^p q_{ij}\left(\frac{\lambda_j}{\lambda_j^2 + \rho_n} + \frac{\rho_n\lambda_j}{(\lambda_j^2 + \rho_n)^2}\right)p_{lj}\right)^2\\
= \max_{i = 1,2,...,p}\sum_{j = 1}^pq^2_{ij}\left(\frac{\lambda_j}{\lambda^2_j + \rho_n} + \frac{\rho_n\lambda_j}{(\lambda^2_j + \rho_n)^2}\right)^2\leq \frac{4}{\lambda^2_p}
\end{aligned}
\end{equation}
Define $s = (s_1,...,s_p)^T = Q^T\beta$, from lemma \ref{lemmaMoment}
\begin{equation}
\begin{aligned}
\widetilde{\beta}_i - \beta_i = -\rho^2_n\sum_{j = 1}^p \frac{q_{ij}s_j}{(\lambda^2_j + \rho_n)^2} + \sum_{j = 1}^p\sum_{l = 1}^n q_{ij}\left(\frac{\lambda_j}{\lambda_j^2 + \rho_n} + \frac{\rho_n\lambda_j}{(\lambda_j^2 + \rho_n)^2}\right)p_{lj}\epsilon_l\\
\Rightarrow \max_{i = 1,...,p}\vert\widetilde{\beta}_i - \beta_i\vert\\
\leq \max_{i = 1,...,p}\sqrt{\sum_{j = 1}^pq^2_{ij}}\times \frac{\rho^2_n\vertiii{\beta}_2}{\lambda^4_p} + \max_{i = 1,2,...,p}\vert\sum_{j = 1}^p\sum_{l = 1}^n q_{ij}\left(\frac{\lambda_j}{\lambda_j^2 + \rho_n} + \frac{\rho_n\lambda_j}{(\lambda_j^2 + \rho_n)^2}\right)p_{lj}\epsilon_l\vert\\
= O_p\left(n^{1/m-\eta}\right)
\end{aligned}
\label{betaS}
\end{equation}
Therefore,
\begin{equation}
\begin{aligned}
Prob\left(\widehat{\mathcal{N}}_{b_n}\neq \mathcal{N}_{b_n}\right)\leq Prob\left(\min_{i\in\mathcal{N}_{b_n}}\vert\widetilde{\beta}_i\vert\leq b_n\right) + Prob\left(\max_{i\not\in\mathcal{N}_{b_n}}\vert\widetilde{\beta}_i\vert > b_n\right)\\
\leq Prob\left(\min_{i\in\mathcal{N}_{b_n}}\vert\beta_i\vert - \max_{i\in\mathcal{N}_{b_n}}\vert\widetilde{\beta}_i - \beta_i\vert\leq b_n\right)\\
+ Prob\left(\max_{i\not\in\mathcal{N}_{b_n}}\vert\beta_i\vert+\max_{i\not\in\mathcal{N}_{b_n}}\vert\widetilde{\beta}_i - \beta_i\vert>b_n\right)
\end{aligned}
\end{equation}
From assumption 4, $\min_{i\in\mathcal{N}_{b_n}}\vert\beta_i\vert-b_n\geq (1/c_b-1)b_n$ and $b_n - \max_{i\not\in\mathcal{N}_{b_n}}\vert\beta_i\vert\geq (1-c_b)b_n$.  So $Prob\left(\widehat{\mathcal{N}}_{b_n}\neq \mathcal{N}_{b_n}\right) = o(1)$ is proved by \eqref{betaS}.

If $\widehat{\mathcal{N}}_{b_n} = \mathcal{N}_{b_n}$, we define $\{c_{ij}\}_{i = 1,...,p_1,j=1,...,p}$ as in section \ref{Prelim} with $b = b_n$ and $\gamma =(\gamma_1,...,\gamma_{p_1})^T= M\beta$. Then
\begin{equation}
\begin{aligned}
\widehat{\gamma}_i - \gamma_i  =\sum_{j\in\mathcal{N}_{b_n}}m_{ij}(\widetilde{\beta}_j - \beta_j) - \sum_{j\not\in\mathcal{N}_{b_n}}m_{ij}\beta_j\\
= -\rho^2_n\sum_{j = 1}^p\frac{c_{ij}s_j}{(\lambda^2_j + \rho_n)^2} + \sum_{j = 1}^p\sum_{l = 1}^n c_{ij}\left(\frac{\lambda_j}{\lambda_j^2 + \rho_n} + \frac{\rho_n\lambda_j}{(\lambda_j^2 + \rho_n)^2}\right)p_{lj}\epsilon_l - \sum_{j\not\in\mathcal{N}_{b_n}}m_{ij}\beta_j
\label{Mbeta}
\end{aligned}
\end{equation}
For $\max_{i = 1,...,p_1}\vert\rho^2_n\sum_{j = 1}^p\frac{c_{ij}s_j}{(\lambda^2_j + \rho_n)^2}\vert\leq \frac{\rho^2_n \sqrt{C_{\mathcal{M}}}}{\lambda^4_p}\times \vertiii{\beta}_2 = O(n^{\alpha_\beta-2\delta})$, and
\begin{equation}
\begin{aligned}
\max_{i = 1,2,...,p_1}\sum_{l = 1}^n \left(\sum_{j = 1}^p c_{ij}\left(\frac{\lambda_j}{\lambda_j^2 + \rho_n} + \frac{\rho_n\lambda_j}{(\lambda_j^2 + \rho_n)^2}\right)p_{lj}\right)^2\\
 = \max_{i = 1,2,...,p_1}\sum_{j = 1}^p c^2_{ij}\left(\frac{\lambda_j}{\lambda_j^2 + \rho_n} + \frac{\rho_n\lambda_j}{(\lambda_j^2 + \rho_n)^2}\right)^2\leq\frac{4C_{\mathcal{M}}}{\lambda^2_p}
\end{aligned}
\end{equation}
lemma \ref{lemmaMoment} implies $\max_{i = 1,2,...,p_1}\vert\widehat{\gamma}_i - \gamma_i\vert = O_p(n^{-\eta})$.

If $\widehat{\mathcal{N}}_{b_n} = \mathcal{N}_{b_n}$, we define $u_{ij} = \sum_{k\in\mathcal{N}_{b_n}}x_{ik}q_{kj}, i = 1,...,n, j = 1,...,p$,
\begin{equation}
\begin{aligned}
\widehat{\epsilon}_i - \epsilon_i  = -\sum_{j\in\mathcal{N}_{b_n}}x_{ij}(\widetilde{\beta}_j - \beta_j) +  \sum_{j\not\in\mathcal{N}_{b_n}}x_{ij}\beta_j
\\ = \rho_n^2\sum_{j = 1}^p \frac{u_{ij}s_j}{(\lambda^2_j + \rho_n)^2} - \sum_{j=1}^p\sum_{l=1}^nu_{ij}\left(\frac{\lambda_j}{\lambda_j^2 + \rho_n} + \frac{\rho_n\lambda_j}{(\lambda_j^2 + \rho_n)^2}\right)p_{lj}\epsilon_l + \sum_{j\not\in\mathcal{N}_{b_n}}x_{ij}\beta_j
\end{aligned}
\end{equation}
For $\max_{i =1,...,n}\rho_n^2\vert\sum_{j = 1}^p \frac{u_{ij}s_j}{(\lambda^2_j + \rho_n)^2}\vert\leq \frac{\rho^2_n\vertiii{\beta}_2}{\lambda^4_p}\times \max_{i=1,..,n}\sqrt{\sum_{j = 1}^pu^2_{ij}} =O(n^{\alpha_\beta-2\delta}\sqrt{\vert\mathcal{N}_{b_n}\vert})$, and
\begin{equation}
\begin{aligned}
\max_{i = 1,2,...,n}\sum_{l = 1}^n\left(\sum_{j = 1}^p u_{ij}\left(\frac{\lambda_j}{\lambda_j^2 + \rho_n} + \frac{\rho_n\lambda_j}{(\lambda_j^2 + \rho_n)^2}\right)p_{lj}\right)^2\\
\leq \frac{4}{\lambda^2_p}\max_{i = 1,2,...,n} \sum_{j\in\mathcal{N}_{b_n}}x^2_{ij} = O\left(\vert\mathcal{N}_{b_n}\vert\times n^{-2\eta}\right)
\end{aligned}
\end{equation}
from lemma \ref{lemmaMoment},
\begin{equation}
\max_{i =1,...,n}\vert\widehat{\epsilon}_i-\epsilon_i\vert  = o_p(n^{(1/m-\eta)/2}) + O_p(n^{\alpha_\beta-2\delta}\sqrt{\vert\mathcal{N}_{b_n}\vert}) +O_p(n^{1/m-\eta}\times \sqrt{\vert\mathcal{N}_{b_n}\vert})
\end{equation}
and we prove \eqref{ConsHalf}.
\end{proof}

\begin{proof}[proof of theorem \ref{Central}]
Define $v_{il} = \frac{1}{\tau_i}\sum_{j = 1}^p c_{ij}\left(\frac{\lambda_j}{\lambda^2_j + \rho_n} + \frac{\rho_n\lambda_j}{(\lambda^2_j + \rho_n)^2}\right)p_{lj}$, then $v_{il} = 0$ if $i\not\in\mathcal{M}$. If $\widehat{\mathcal{N}}_{b_n} = \mathcal{N}_{b_n}$, then $\widehat{\tau}_i  =\tau_i$. From \eqref{Mbeta} and Cauchy-Schwarz inequality,
\begin{equation}
\begin{aligned}
\max_{i = 1,...,p_1}\vert\frac{1}{\widehat{\tau}_i}\left(\widehat{\gamma}_i - \sum_{j = 1}^p m_{ij}\beta_j\right) - \sum_{l = 1}^n v_{il}\epsilon_l\vert\\
\leq \max_{i\in\mathcal{M}}\frac{\rho^2_n}{\tau_i}\vert\sum_{j = 1}^p \frac{c_{ij}s_j}{(\lambda^2_j+\rho_n)^2}\vert + \max_{i =1,...,p_1}\frac{1}{\tau_i}\vert\sum_{j\not\in\mathcal{N}_{b_n}}m_{ij}\beta_j\vert\\
\leq\max_{i\in\mathcal{M}}\frac{\rho^2_n}{\tau_i}\sqrt{\sum_{j=1}^p \frac{c^2_{ij}\lambda^2_j}{(\lambda^2_j + \rho_n)^2}}\sqrt{\sum_{j = 1}^p \frac{s^2_j}{\lambda^2_j(\lambda^2_j + \rho_n)^2}}  + \sqrt{n}\max_{i =1,...,p_1}\vert\sum_{j\not\in\mathcal{N}_{b_n}}m_{ij}\beta_j\vert\\
\leq \frac{\rho^2_n\vertiii{\beta}_2}{\lambda^3_p}+\sqrt{n}\max_{i =1,...,p_1}\vert\sum_{j\not\in\mathcal{N}_{b_n}}m_{ij}\beta_j\vert
\end{aligned}
\label{realCov}
\end{equation}
which has order $O(n^{\alpha_\beta+\eta-2\delta}) + o(1)$. Besides, if $i\in\mathcal{M}$, then for sufficiently large $n$
\begin{equation}
\begin{aligned}
\sum_{l=1}^n v^2_{il} = \frac{1}{\tau^2_i}\sum_{j=1}^p c^2_{ij}\left(\frac{\lambda_j}{\lambda^2_j + \rho_n} + \frac{\rho_n\lambda_j}{(\lambda_j^2 + \rho_n)^2}\right)^2\leq 1\\
\sum_{l=1}^n v^2_{il} = \frac{1}{1 + \frac{1}{n\sum_{j = 1}^pc^2_{ij}\left(\frac{\lambda_j}{\lambda^2_j + \rho_n} + \frac{\rho_n\lambda_j}{(\lambda^2_j + \rho_n)^2}\right)^2}}
\geq \frac{1}{1+\frac{1}{n\sum_{j = 1}^pc^2_{ij}/(4\lambda^2_j)}}\geq \frac{1}{1 + \frac{4C^2_\lambda}{c_{\mathcal{M}}}}>0
\label{Vs}
\end{aligned}
\end{equation}
Assumption 6 implies that the matrix $\{v_{il}\}_{i\in\mathcal{M},l = 1,...,n}$ has rank $\vert\mathcal{M}\vert$. Define $t = \frac{\rho^2_n\vertiii{\beta}_2}{\lambda^3_p}+\max_{i =1,...,p_1}\sqrt{n}\vert\sum_{j\not\in\mathcal{N}_{b_n}}m_{ij}\beta_j\vert$,
lemma \ref{GaussianRes} and lemma \ref{GaussianApprox} imply
\begin{equation}
\begin{aligned}
Prob\left(\max_{i = 1,...,p_1}\frac{\vert\widehat{\gamma}_i - \sum_{j = 1}^p m_{ij}\beta_j\vert}{\widehat{\tau}_i}\leq x\right) - Prob\left(\max_{i\in\mathcal{M}}\vert\sum_{l=1}^nv_{il}\xi_l\vert\leq x\right)\\
\leq Prob\left(\widehat{\mathcal{N}}_{b_n}
\neq \mathcal{N}_{b_n}\right)
+ \vert Prob\left(\max_{i\in\mathcal{M}}\vert\sum_{l = 1}^n v_{il}\epsilon_l\vert\leq x + t\right) - Prob\left(\max_{i\in\mathcal{M}}\vert\sum_{l = 1}^nv_{il}\xi_l\vert\leq x+t\right)\vert\\
+ Ct(1+\sqrt{\log(p_1)}+\sqrt{\vert\log(t)\vert})\\
Prob\left(\max_{i = 1,...,p_1}\frac{\vert\widehat{\gamma}_i - \sum_{j = 1}^p m_{ij}\beta_j\vert}{\widehat{\tau}_i}\leq x\right) - Prob\left(\max_{i\in\mathcal{M}}\vert\sum_{l=1}^nv_{il}\xi_l\vert\leq x\right)\\
\geq -Prob\left(\widehat{\mathcal{N}}_{b_n}\neq \mathcal{N}_{b_n}\right)
-\vert Prob\left(\max_{i\in\mathcal{M}}\vert\sum_{l = 1}^n v_{il}\epsilon_l\vert\leq x - t\right) - Prob\left(\max_{i\in\mathcal{M}}\vert\sum_{l = 1}^nv_{il}\xi_l\vert\leq x - t\right)\vert\\
 - Ct(1+\sqrt{\log(p_1)}+\sqrt{\vert\log(t)\vert})\\
\Rightarrow \sup_{x\in\mathbf{R}}\vert Prob\left(\max_{i = 1,...,p_1}\frac{\vert\widehat{\gamma}_i - \sum_{j = 1}^p m_{ij}\beta_j\vert}{\widehat{\tau}_i}\leq x\right) - Prob\left(\max_{i\in\mathcal{M}}\vert\sum_{l=1}^nv_{il}\xi_l\vert\leq x\right)\vert\\
\leq Prob\left(\widehat{\mathcal{N}}_{b_n}\neq \mathcal{N}_{b_n}\right)
 + \sup_{x\in\mathbf{R}}\vert Prob\left(\max_{i\in\mathcal{M}}\vert\sum_{l = 1}^n v_{il}\epsilon_l\vert\leq x\right) - Prob\left(\max_{i\in\mathcal{M}}\vert\sum_{l = 1}^nv_{il}\xi_l\vert\leq x\right)\vert\\
 + Ct(1+\sqrt{\log(p_1)}+\sqrt{\vert\log(t)\vert})
\end{aligned}
\end{equation}
and we prove \eqref{Cent}. Here $C$ is a constant.
\end{proof}
\section{Proof of theorem \ref{thmBootS}}
\label{app3}
Errors in the linear model \eqref{linearModel} can be dependent, non-stationary and heteroskedastic. So it is hopeless to estimate the errors' variance and covariance. However, it is still possible to make a consistent estimation of the estimator $\widehat{\gamma} = M\widehat{\beta}$'s variances.
Lemma \ref{lemmaCov} justifies this idea.
\begin{lemma}
Suppose random variables $\epsilon_i, i = 1,2,...,n$ satisfy assumption 2; $(\gamma_{ij})_{i =1,...,p_1, j =1,...,n}\in\mathbf{R}^{p_1\times n}$ satisfy conditions in lemma \ref{GaussianApprox}; the kernel function $K$ satisfies assumption 7;
and $k_n > 0$ is a chosen bandwidth. Then
\begin{equation}
\begin{aligned}
\max_{i,j = 1,...,p_1}\vert\sum_{s_1 = 1}^n\sum_{s_2 = 1}^n\gamma_{is_1}\gamma_{js_2}K\left(\frac{s_1-s_2}{k_n}\right)\epsilon_{s_1}\epsilon_{s_2} - \sum_{s_1 = 1}^n\sum_{s_2=1}^n\gamma_{is_1}\gamma_{js_2}\sigma_{s_1s_2}\vert\\
 = o_p(k_n\times n^{-1/4}\log^{-z}(n)) + O_p(v_n)
\end{aligned}
\label{Cov}
\end{equation}
Here $\sigma_{s_1s_2}, s_1, s_2 = 1,...,n$ are defined in section \ref{Prelim} and $z = \max(\frac{9}{2}, \frac{3\alpha_\epsilon}{2\alpha_\epsilon-2})$,
\begin{equation}
v_n =
\begin{cases}
k_n^{1-\alpha_\epsilon}\ \text{if } 1<\alpha_\epsilon<2\\
\log(k_n)/k_n\ \text{if } \alpha_\epsilon = 2\\
1/k_n\ \text{if }\alpha_\epsilon>2
\end{cases}
\end{equation}
\label{lemmaCov}
\end{lemma}
\begin{proof}
From \eqref{Decay} and section 0.9.7 in \cite{matrix}, there exists a constant $C>0$ such that
\begin{equation}
\begin{aligned}
\vert\sum_{s_1 = 1}^n\sum_{s_2=1}^n\left(1 - K\left(\frac{s_1 - s_2}{k_n}\right)\right)\gamma_{is_1}\gamma_{js_2}\sigma_{s_1s_2}\vert\\
\leq C\sum_{s_1 =1}^n\sum_{s_2 = 1}^n\vert\gamma_{is_1}\gamma_{js_2}\vert\times\frac{\left(1 - K\left(\frac{s_1 - s_2}{k_n}\right)\right)}{(1 + \vert s_1 - s_2\vert)^{\alpha_\epsilon}}\\
\leq 2C\sqrt{\sum_{s=1}^n\gamma^2_{is}}\times\sqrt{\sum_{s=1}^n\gamma^2_{js}}\times \sum_{s = 0}^{\infty}\frac{1 - K(s/k_n)}{1 + s^{\alpha_\epsilon}}
\end{aligned}
\end{equation}
$K$ is continuously differentiable, so
\begin{equation}
\begin{aligned}
\sum_{s = 0}^{\infty}\frac{1 - K(s/k_n)}{1 + s^{\alpha_\epsilon}}\leq \frac{\max_{x\in[0,1]}\vert K^{'}(x)\vert}{k_n}\sum_{s = 0}^{k_n}\frac{s}{1 + s^{\alpha_\epsilon}} + \sum_{s = k_n+1}^\infty\frac{1}{1 + s^{\alpha_\epsilon}}
\\
= O\left(\frac{1}{k_n}\left(1 + \int_{[1,k_n]}x^{1-\alpha_\epsilon}dx\right) + \int_{[k_n,\infty)}x^{-\alpha_\epsilon}dx\right)
\end{aligned}
\label{firHalf}
\end{equation}
we have $\max_{i,j = 1,...,p_1}\vert\sum_{s_1 = 1}^n\sum_{s_2=1}^n\left(1 - K\left(\frac{s_1 - s_2}{k_n}\right)\right)\gamma_{is_1}\gamma_{js_2}\sigma_{s_1s_2}\vert = O(v_n)$.

On the other hand, define $\zeta_{i,k} = \epsilon_i\epsilon_{i+k} - \sigma_{ii+k} = h_{i,i+k}(...,e_{i+k-1}, e_{i+k})$(i.e., $\zeta_{i,k}$ is $\mathcal{F}_{i+k}$ measurable), and define $\zeta_{i,k,t} = h_{i,i+k}(...,e_{i+k-t-1}, e_{i+k-t}^\dagger, e_{i+k-t+1},...,e_{i+k})$. Here $k\geq 0, i\geq 1, i+k\leq n, t\geq 0$. $e_i, e^\dagger_i, \sigma_{ij}$ are defined in section \ref{Prelim}. Then $\mathbf{E}\zeta_{i,k} = 0$ and
\begin{equation}
\begin{aligned}
\psi_{i,k,t,m/2} = \Vert\zeta_{i,k} - \zeta_{i,k,t}\Vert_{m/2} = \Vert\epsilon_i\epsilon_{i+k} - \epsilon_{i,t-k}\epsilon_{i+k,t}\Vert_{m/2}\\
\leq \Vert\epsilon_i - \epsilon_{i,t-k}\Vert_m \Vert\epsilon_{i+k}\Vert_m +\Vert\epsilon_{i,t-k}\Vert_m\Vert\epsilon_{i+k} - \epsilon_{i+k,t}\Vert_m\\
\Rightarrow \max_{i = 1,...,n-k}\psi_{i,k,t,m/2}\leq C\max_{i =1,...,n}\delta_{i, t-k, m} + C\max_{i =1,...,n}\delta_{i,t,m}\ \text{if $t-k\geq 0$}\\
\text{ and }C\max_{i =1,...,n}\delta_{i,t,m}\ \text{if $t-k<0$}
\end{aligned}
\end{equation}
Here $\epsilon_{i,t} =\epsilon_i$ if $t<0$ and $C$ is a constant. For fixed $i,j$, define

\noindent$N_{s,k,t} = \sum_{l = n-k+1-s}^{n-k}\gamma_{il}\gamma_{jl+k}\left(\mathbf{E}\zeta_{l,k}|\mathcal{F}_{l+k,t} - \mathbf{E}\zeta_{l,k}|\mathcal{F}_{l+k, t-1}\right)$, then $N_{s,k,t}$ is $\mathcal{F}_{n, s+t-1}$
measurable. Apply $\pi-\lambda$ theorem to the $\lambda-$ system
\begin{equation}
\{A\in\mathcal{F}_{n, s + t - 1}\Big|\mathbf{E}\left(\mathbf{E}\zeta_{n-k-s,k}|\mathcal{F}_{n-s, t}\right)\times\mathbf{1}_A = \mathbf{E}\left(\mathbf{E}\zeta_{n-k-s,k}|\mathcal{F}_{n-s, t - 1}\right)\times\mathbf{1}_A\}
\end{equation}
and the $\pi-$system $\{A_n\times A_{n-1}\times...\times A_{n-s-t+1}\}$, $A_i$ is generated by $e_i$. Then we know that $N_{s,k,t}, s = 1,2,...,n-k$ form a martingale for any given $k,t$. From \eqref{Deltas}, \eqref{DeltaM} and \eqref{LinearComb}
\begin{equation}
\begin{aligned}
\Vert\sum_{l = 1}^{n-k}\gamma_{il}\gamma_{jl+k}\zeta_{l,k}\Vert_{m/2}\leq \Vert\sum_{l = 1}^{n-k}\gamma_{il}\gamma_{jl+k}\mathbf{E}\zeta_{l,k}|\mathcal{F}_{l+k,0}\Vert_{m/2} + \sum_{t = 1}^\infty\Vert N_{n-k,k,t}\Vert_{m/2}\\
\leq C\sqrt{\sum_{l = 1}^{n-k}\gamma^2_{il}\gamma_{jl+k}^2\Vert\zeta_{l,k}\Vert_{m/2}^2} + C\sqrt{\sum_{l = 1}^{n-k}\gamma^2_{il}\gamma_{jl+k}^2}\times\sum_{t = 1}^\infty\max_{l = 1,...,n-k}\psi_{l,k,t, m/2}
\end{aligned}
\end{equation}
Here $C$ is a constant independent with $i,j$. Since $\Vert\zeta_{l,k}\Vert_{m/2}\leq \Vert\epsilon_i\Vert_m\times\Vert\epsilon_{i+k}\Vert_m+\vert\sigma_{ii+k}\vert$ and
\begin{equation}
\sum_{t = 1}^\infty \max_{l = 1,...,n-k}\psi_{l,k,t, m/2}\leq 2C\sum_{t=0}^\infty \max_{i =1,...,n}\delta_{i,t,m}\ \text{with $C$ a constant}
\end{equation}
we have
\begin{equation}
\max_{i,j = 1,...,p_1}\frac{\Vert\sum_{l = 1}^{n-k}\gamma_{il}\gamma_{jl+k}\zeta_{l,k}\Vert_{m/2}}{\sqrt{\sum_{l = 1}^{n-k}\gamma^2_{il}\gamma^2_{jl+k}}} = O(1)
\end{equation}
Therefore
\begin{equation}
\begin{aligned}
\Vert\sum_{s_1 = 1}^n\sum_{s_2 = 1}^n\gamma_{is_1}\gamma_{js_2}K\left(\frac{s_1-s_2}{k_n}\right)\times(\epsilon_{s_1}\epsilon_{s_2}-\sigma_{s_1s_2})\Vert_{m/2}\\
\leq 2\sum_{l = 0}^{n-1}K\left(\frac{l}{k_n}\right)\Vert\sum_{s = 1}^{n-l}\gamma_{is}\gamma_{js+l}\zeta_{s,l}\Vert_{m/2}\\
\leq C\sum_{l = 0}^\infty K\left(\frac{l}{k_n}\right)\times\max_{i =1,...,p_1, j = 1,...,n}\vert\gamma_{ij}\vert\ \text{with $C$ a constant}
\end{aligned}
\label{secHalf}
\end{equation}
For $K$ is decreasing on $[0,\infty)$,
\begin{equation}
\sum_{l = 1}^\infty K(l/k_n)\leq \sum_{l = 1}^\infty \int_{[l-1,l]}K(x/k_n)dx = k_n\int_{[0,\infty)}K(x)dx = O(k_n) \label{Ks}
\end{equation}
From \eqref{firHalf} and \eqref{secHalf}, we prove \eqref{Cov}.
\end{proof}

\begin{proof}[proof of theorem \ref{thmBootS}]
Recall $\mathbf{E}^*\cdot = \mathbf{E}\cdot|y$, $Prob^*(\cdot) = Prob(\cdot|y)$, and define $\varepsilon^*_i,i=1,...,n$ as in algorithm \ref{alg1}. $\forall a = (a_1,...,a_n)^T\in\mathbf{R}^n$, $\sum_{i = 1}^n a_i\varepsilon^*_i$ has normal distribution. Section 0.9.7 in \cite{matrix} and \eqref{Ks} implies
\begin{equation}
\mathbf{E}^*\vert\sum_{i=1}^na_i\varepsilon^*_i\vert^m = C\left(\sum_{i = 1}^n\sum_{j = 1}^n a_ia_jK\left(\frac{i-j}{k_n}\right)\right)^{m/2} = O\left(k_n^{m/2}\times \vertiii{a}_2^m\right)
\end{equation}
Here $C = \mathbf{E}|Y|^m$, $Y$ has normal distribution with mean $0$ and variance $1$. Therefore, $\forall \xi>0$,
\begin{equation}
\begin{aligned}
Prob^*\left(\max_{i = 1,...,p}\vert\sum_{j=1}^p\sum_{l=1}^n q_{ij}p_{lj}\left(\frac{\lambda_j}{\lambda^2_j+\rho_n} + \frac{\rho_n\lambda_j}{(\lambda^2_j+\rho_n)^2}\right)\epsilon^*_l\vert>\xi\right)\\
\leq \frac{1}{\xi^m}\sum_{i=1}^p\mathbf{E}^*\vert\sum_{j=1}^p\sum_{l=1}^n q_{ij}p_{lj}\left(\frac{\lambda_j}{\lambda^2_j+\rho_n} + \frac{\rho_n\lambda_j}{(\lambda^2_j+\rho_n)^2}\right)\epsilon^*_l\vert^m\\
=\frac{Ck_n^{m/2}}{\xi^m}\sum_{i=1}^p\left(\sum_{l = 1}^n\widehat{\epsilon}_l^2\left(\sum_{j = 1}^p q_{ij}p_{lj}\left(\frac{\lambda_j}{\lambda^2_j+\rho_n}+\frac{\rho_n\lambda_j}{(\lambda_j^2+\rho_n)^2}\right)\right)^2\right)^{m/2}\\
\leq \frac{Ck_n^{m/2}}{\xi^m}\times \max_{l=1,...,n}\vert\widehat{\epsilon}_l\vert^m\times \frac{2^m p}{\lambda^m_p}
\end{aligned}
\label{Qstar}
\end{equation}
Here $C$ is a constant. From theorem \ref{thm1}, $\max_{i = 1,...,n}\vert\widehat{\epsilon}_i\vert^m\leq 2^m\max_{i = 1,...,n}\vert\epsilon_i\vert^m + 2^m\max_{i = 1,...,n}\vert\widehat{\epsilon}_i-\epsilon_i\vert^m = O_p(n)$.
If $\widehat{\mathcal{N}}_{b_n} = \mathcal{N}_{b_n}$, \eqref{betaS} implies $\vertiii{\widehat{\beta}}_2\leq 2\sqrt{\sum_{i\in\mathcal{N}_{b_n}}(\widetilde{\beta}_i-\beta_i)^2} + 2\vertiii{\beta}_2 = O_p(n^{\alpha_\beta})$. Define $\widehat{s} = (\widehat{s}_1,...,\widehat{s}_p)^T = Q^T\widehat{\beta}$. If $\widehat{\mathcal{N}}_{b_n} = \mathcal{N}_{b_n}$, then for any given $0<a<1$, there exists a constant $C > 0$ such that $\vertiii{\widehat{\beta}}_2\leq Cn^{\alpha_\beta}$ and $\max_{l = 1,...,n}\vert\widehat{\epsilon}_i\vert^m\leq Cn$ with probability at least $1 - a$. From \eqref{Qstar}, $\forall c>0$,
\begin{equation}
\begin{aligned}
\widetilde{\beta}^*_i - \widehat{\beta}_i = -\rho^2_n\sum_{j=1}^p\frac{q_{ij}\widehat{s}_j}{(\lambda^2_j+\rho_n)^2} + \sum_{j=1}^p\sum_{l=1}^n q_{ij}p_{lj}\left(\frac{\lambda_j}{\lambda^2_j+\rho_n} + \frac{\rho_n\lambda_j}{(\lambda^2_j+\rho_n)^2}\right)\epsilon^*_l\\
\Rightarrow Prob^*\left(\max_{i=1,...,p}\vert\widetilde{\beta}^*_i - \widehat{\beta}_i\vert> 2cn^{-\nu_b}\right)\\
\leq Prob^*\left(\frac{n^{\nu_b}\rho^2_n\vertiii{\widehat{\beta}}_2}{\lambda^4_p}>c\right) + \frac{C^{'}}{c^m}\times\left(\sqrt{k_n}\times n^{2/m-\eta+\nu_b}\right)^m\\
\end{aligned}
\label{ConsisBoot}
\end{equation}
Here $C^{'}$ depends on $C$. If $\widehat{\mathcal{N}}_{b_n} = \mathcal{N}_{b_n}$,
\begin{equation}
\begin{aligned}
Prob^*\left(\widehat{\mathcal{N}}_{b_n}^*\neq \mathcal{N}_{b_n}\right)\leq Prob^*\left(\max_{i\not\in\mathcal{N}_{b_n}}\vert\widetilde{\beta}^*_i\vert > b_n\right) + Prob^*\left(\min_{i\in\mathcal{N}_{b_n}}\vert\widetilde{\beta}^*_i\vert\leq b_n\right)\\
\leq Prob^*\left(\max_{i\not\in\mathcal{N}_{b_n}}\vert\widetilde{\beta}^*_i - \widehat{\beta}_i\vert>b_n - \max_{i\not\in\mathcal{N}_{b_n}}\vert\widetilde{\beta}_i\vert\right) + Prob^*\left(\max_{i\in\mathcal{N}_{b_n}}\vert\widetilde{\beta}^*_i - \widehat{\beta}_i\vert\geq \min_{i\in\mathcal{N}_{b_n}}\vert\widetilde{\beta}_i\vert - b_n\right)
\end{aligned}
\end{equation}
From assumption 4 and \eqref{betaS}, with probability tending to $1$, $\max_{i\not\in\mathcal{N}_{b_n}}\vert\widetilde{\beta}_i\vert \leq \frac{1+c_b}{2}b_n$ and $\min_{i\in\mathcal{N}_{b_n}}\vert\widetilde{\beta}_i\vert\geq \frac{2b_n}{1+c_b}$. \eqref{ConsisBoot} implies $Prob^*\left(\widehat{\mathcal{N}}_{b_n}^*\neq \mathcal{N}_{b_n}\right) = o_p(1)$.

If $\widehat{\mathcal{N}}_{b_n}^* = \widehat{\mathcal{N}}_{b_n} =\mathcal{N}_{b_n}$, define $c_{ij}, i = 1,...,p_1, j = 1,...,p$ as in section \ref{Prelim},
\begin{equation}
\begin{aligned}
\widehat{\gamma}^*_i - \widehat{\gamma_i} = \sum_{j\in\mathcal{N}_{b_n}}m_{ij}(\widetilde{\beta}^*_j - \widehat{\beta}_j)\\
= -\rho_n^2\sum_{j = 1}^p \frac{c_{ij}\widehat{s}_j}{(\lambda^2_j+\rho_n)^2} + \sum_{j = 1}^p\sum_{l = 1}^n c_{ij}p_{lj}\left(\frac{\lambda_j}{\lambda^2_j+\rho_n} + \frac{\rho_n\lambda_j}{(\lambda^2_j+\rho_n)^2}\right)\epsilon^*_l
\end{aligned}
\label{bootAppli}
\end{equation}
Besides, $\widehat{\tau}^*_i = \tau_i$, and $\widehat{c}^*_{ij} = c_{ij}, i = 1,...,p_1,j=1,...,p$. Define $v_{il} = \frac{1}{\tau_i}\sum_{j=1}^pc_{ij}p_{lj}\left(\frac{\lambda_j}{\lambda_j^2+\rho_n} + \frac{\rho_n\lambda_j}{(\lambda^2_j+\rho_n)^2}\right), i = 1,...,p_1, l = 1,...,n$; and $z = (z_1,...,z_{p_1})^T$ such that $z_i = \sum_{l=1}^nv_{il}\epsilon^*_l$. We have $\mathbf{E}^*z_iz_j = \sum_{l_1 = 1}^n\sum_{l_2 = 1}^n v_{il_1}v_{jl_2}\widehat{\epsilon}_{l_1}\widehat{\epsilon}_{l_2}K\left(\frac{l_1 - l_2}{k_n}\right)$. Form section 0.9.7 in \cite{matrix}, lemma \ref{lemmaCov}, and theorem \ref{thm1}, $\forall i,j\in\mathcal{M}$,
\begin{equation}
\begin{aligned}
\vert\mathbf{E}^*z_iz_j  - \sum_{l_1=  1}^n\sum_{l_2 = 1}^n v_{il_1}v_{jl_2}\sigma_{l_1l_2}\vert\leq \vert\sum_{l_1=  1}^n\sum_{l_2= 1}^nv_{il_1}v_{jl_2}K\left(\frac{l_1-l_2}{k_n}\right)\widehat{\epsilon}_{l_1}(\widehat{\epsilon}_{l_2}-\epsilon_{l_2})\vert\\
+\vert\sum_{l_1=  1}^n\sum_{l_2= 1}^nv_{il_1}v_{jl_2}K\left(\frac{l_1-l_2}{k_n}\right)(\widehat{\epsilon}_{l_1}-\epsilon_{l_1})\epsilon_{l_2}\vert\\
+\vert\sum_{l_1 = 1}^n\sum_{l_2 = 1}^nv_{il_1}v_{jl_2}\epsilon_{l_1}\epsilon_{l_2}K\left(\frac{l_1-l_2}{k_n}\right) - \sum_{l_1= 1}^n\sum_{l_2 =1}^nv_{il_1}v_{jl_2}\sigma_{l_1l_2}\vert\\
\leq2\sum_{l = 0}^\infty K\left(\frac{l}{k_n}\right)\times\left(\sqrt{\sum_{l=1}^n v^2_{il}\widehat{\epsilon}_l^2}\sqrt{\sum_{l=1}^nv^2_{jl}(\widehat{\epsilon}_{l}-\epsilon_{l})^2}+\sqrt{\sum_{l=1}^nv^2_{il}(\widehat{\epsilon}_l-\epsilon_l)^2}\sqrt{\sum_{l=1}^nv^2_{jl}\epsilon^2_l}\right)\\
+\vert\sum_{l_1 = 1}^n\sum_{l_2 = 1}^nv_{il_1}v_{jl_2}\epsilon_{l_1}\epsilon_{l_2}K\left(\frac{l_1-l_2}{k_n}\right) - \sum_{l_1= 1}^n\sum_{l_2 =1}^nv_{il_1}v_{jl_2}\sigma_{l_1l_2}\vert\\
 = o_p(k_n\times n^{\frac{3}{2m}-\frac{\eta}{2}}) + o_p(k_n\times n^{-1/4}\log^{-z}(n)) + O_p(v_n) = o_p(1)
\end{aligned}
\label{ErrorTerm}
\end{equation}

$v_n$ is defined in lemma \ref{lemmaCov} and $z$ is defined in assumption 6. From lemma \ref{GaussianRes}, assumption 2, assumption 6, \eqref{LinearComb}, and \eqref{Vs}, $\sup_{x\in\mathbf{R}}\vert Prob^*\left(\max_{i\in\mathcal{M}}\vert z_i\vert\leq x\right) - H(x)\vert = o_p(1)$. Since

\begin{equation}
\begin{aligned}
\max_{i = 1,...,p_1}\frac{\rho_n^2}{\tau_i}\vert\sum_{j = 1}^p \frac{c_{ij}\widehat{s}_j}{(\lambda^2_j+\rho_n)^2}\vert
\leq \max_{i\in\mathcal{M}}\frac{\rho_n^2}{\tau_i}\sqrt{\sum_{j = 1}^p\frac{c^2_{ij}\lambda_j^2}{(\lambda^2_j+\rho_n)^2}}\sqrt{\sum_{j = 1}^p\frac{\widehat{s}^2_j}{\lambda^2_j(\lambda^2_j+\rho_n)^2}}\\
\leq \frac{\rho^2_n\vertiii{\widehat{\beta}}_2}{\lambda^3_p} = o_p(1)
\end{aligned}
\end{equation}
Define $t = \frac{\rho^2_n\vertiii{\widehat{\beta}}_2}{\lambda^3_p}$ and assume $\widehat{\mathcal{N}}_{b_n} = \mathcal{N}_{b_n}$,
\begin{equation}
\begin{aligned}
Prob^*\left(\max_{i = 1,...,p_1}\frac{\vert\widehat{\gamma}^*_i - \widehat{\gamma}_i\vert}{\widehat{\tau}^*_i}\leq x\right) - H(x)\leq Prob^*\left(\widehat{\mathcal{N}}^*_{b_n}\neq \mathcal{N}_{b_n}\right)\\
 + \left(Prob^*\left(\max_{i\in\mathcal{M}}\vert z_i\vert\leq x + t\right) - H(x+t)\right)
 + Ct(1 + \sqrt{\log(p_1)}+\sqrt{\vert\log(t)\vert})\\
 \text{and }
Prob^*\left(\max_{i = 1,...,p_1}\frac{\vert\widehat{\gamma}^*_i - \widehat{\gamma}_i\vert}{\widehat{\tau}^*_i}\leq x\right) - H(x)\geq -Prob^*\left(\widehat{\mathcal{N}}^*_{b_n}\neq \mathcal{N}_{b_n}\right)\\
+\left(Prob^*\left(\max_{i\in\mathcal{M}}\vert z_i\vert\leq x-t\right) - H(x-t)\right) - Ct(1 + \sqrt{\log(p_1)}+\sqrt{\vert\log(t)\vert})\\
\Rightarrow\sup_{x\in\mathbf{R}}\vert Prob^*\left(\max_{i = 1,...,p_1}\frac{\vert\widehat{\gamma}^*_i - \widehat{\gamma}_i\vert}{\widehat{\tau}^*_i}\leq x\right) - H(x)\vert\leq Prob^*\left(\widehat{\mathcal{N}}^*_{b_n}\neq \mathcal{N}_{b_n}\right)\\
+ \sup_{x\in\mathbf{R}}\vert Prob^*\left(\max_{i\in\mathcal{M}}\vert z_i\vert\leq x\right) - H(x)\vert
 + Ct(1 + \sqrt{\log(p_1)}+\sqrt{\vert\log(t)\vert}) = o_p(1)
\end{aligned}
\label{Distri}
\end{equation}

For the non-degenerated joint normal distribution is absolutely continuous with respect to Lebesgure measure, $\forall 0<\alpha<1$, $\exists c_{1-\alpha}\in \mathbf{R}$ such that $H(c_{1-\alpha}) = 1-\alpha$.
$\forall 0<\tau<\min(\alpha/2,(1-\alpha)/2)$, assign $x = c_{1-\alpha+\tau},c_{1-\alpha-\tau}$ in \eqref{Distri}, then $c_{1-\alpha-\tau}\leq c^*_{1-\alpha}\leq c_{1-\alpha+\tau}$ with probability tending to $1$.
Therefore, we have
$Prob\left(\max_{i =1,...,p_1}\frac{1}{\widehat{\tau}_i}\vert\widehat{\gamma}_i - \gamma_i\vert\leq c^*_{1-\alpha}\right)\leq Prob\left(c^*_{1-\alpha}>c_{1-\alpha+\tau}\right) + Prob\left(\max_{i =1,...,p_1}\frac{1}{\widehat{\tau}_i}\vert\widehat{\gamma}_i - \gamma_i\vert\leq c_{1-\alpha+\tau}\right)$;
and $Prob\left(\max_{i =1,...,p_1}\frac{1}{\widehat{\tau}_i}\vert\widehat{\gamma}_i - \gamma_i\vert\leq c^*_{1-\alpha}\right)\geq -Prob(c^*_{1-\alpha}<c_{1-\alpha-\tau}) + Prob(\max_{i =1,...,p_1}\frac{1}{\widehat{\tau}_i}\vert\widehat{\gamma}_i - \gamma_i\vert\leq c_{1-\alpha-\tau})$. By setting $\tau\to 0$, theorem \ref{Central} implies \eqref{thmBoot}.

\end{proof}
\end{document}